\documentclass{ps}
\def\I{{\rm{I\!\!I}}}
\def\leftB{[\![}
\def\rightB{]\!]}

\def\F{{\mathcal{ F}}}

\def\R{{\mathbb{R}} }

\def\N{{\mathbb{N}} }
\def\E{{\mathbb{E}}  }

\def\P{{\mathbb{P}}  }

\def\I{{\mathbb{I}}}
\def\det{{\rm{det}}}
\def\Tr{{\rm{Tr}}}
\def\bint#1^#2{\displaystyle{\int_{#1}^{#2}}}
\def\bsum#1^#2{\displaystyle{\sum_{#1}^{#2}}}

\def\supp{{\rm{supp}}}
\def\xdt_#1{X_#1(\Delta t)}

\newtheorem{THM}{Theorem}
\newtheorem{PROP}{Proposition}
\newtheorem{LEMME}{Lemma}

\newtheorem{REM}{Remark}

\newcommand{\mysection}{\setcounter{equation}{0} \section}

\newcommand \A[1]{{\bf (#1)}}

\begin{document}
\title{Stability of Densities for Perturbed Diffusions and Markov Chains}\thanks{The article was prepared within the framework of a subsidy granted to the HSE by the Government of the Russian Federation for the implementation of the Global Competitiveness Program.
}
\thanks{Support by Deutsche Forschungsgemeinschaft through the Research Training Group RTG 1953 is
gratefully acknowledged by the second author.}

\author{V. Konakov}\address{Higher School of Economics, Shabolovka 31, Moscow, Russian Federation. vkonakov@hse.ru}
\author{A. Kozhina}\address{Higher School of Economics, Shabolovka 31, Moscow, Russian Federation and RTG 1953, Institute of Applied Mathematics, Heidelberg University, Germany. aakozhina@hse.ru}
\author{S. Menozzi}\address{Higher School of Economics, Shabolovka 31, Moscow, Russian Federation and  LaMME, UMR CNRS 8070, Universit\'e d'Evry Val d'Essonne, 23 Boulevard de France, 91037 Evry, France. stephane.menozzi@univ-evry.fr}
\date{\today}
\begin{abstract} 
We study the sensitivity of the densities of non degenerate diffusion processes and related Markov Chains with respect to a perturbation of the coefficients. Natural applications of these results appear in models with misspecified coefficients or for the investigation of the weak error of the Euler scheme with irregular coefficients.
\end{abstract}
\begin{resume} Nous \'etudions la sensibilit\'e des densit\'es de processus de diffusion non d\'eg\'en\'er\'es et des Cha\^{i}nes de Markov associ\'ees par rapport \`a une perturbation des coefficients. Ces r\'esultats trouvent des applications naturelles dans l'\'etude de mod\`eles avec incertitude sur les coefficients ou pour l'analyse de l'erreur faible du sch\'ema d'Euler \`a coefficients irr\'eguliers.  \end{resume}
\maketitle

\mysection{Introduction}
\subsection{Setting.}
For a fixed given deterministic final horizon $T>0$, let us consider the following multidimensional SDE:
\begin{align}
& dX_{t}=b(t,X_{t})dt+\sigma (t,X_{t})dW_{t},\text{ }t\in \lbrack 0,T],
\label{1} 
\end{align}
where $b:[0,T]\times \mathbb{R}^{d}\rightarrow \mathbb{R}^{d},\  \sigma:[0,T]\times \mathbb{R}^{d}\rightarrow \R^d\otimes \R^d$ are bounded coefficients that are measurable in time and H\"older continuous in space (this last condition will be possibly relaxed for the drift term $b$) and 
$W $ is a Brownian motion on some filtered probability space $(\Omega,\F,(\F_t)_{t\ge 0},\P) $. 
Also, $a(t,x):=\sigma\sigma^*(t,x)$ is assumed to be uniformly elliptic.
In particular those assumptions guarantee that \eqref{1} admits a unique weak solution, see e.g. Bass and Perkins \cite{bass:perk:09}, \cite{meno:10} from which the uniqueness to the martingale problem for the associated generator can be derived under the current assumptions.
  
  We now introduce, for a given parameter $\varepsilon>0 $, a perturbed version of \eqref{1} with dynamics:
\begin{align}
& dX_{t}^{(\varepsilon)}=b_{\varepsilon}(t,X_{t}^{(\varepsilon)})dt+\sigma _{\varepsilon}(t,X_{t}^{(\varepsilon)})dW_{t},\text{ }%
t\in \lbrack 0,T],  \label{2}
\end{align}
where $b_\varepsilon:[0,T]\times \mathbb{R}^{d}\rightarrow \mathbb{R}^{d},\  \sigma_\varepsilon:[0,T]\times \mathbb{R}^{d}\rightarrow \R^d\otimes \R^d$ 
satisfy at least the same assumptions as $b,\sigma$ and are in some sense meant to be \textit{close} to $b,\sigma $ when $\varepsilon$ is small.

It is known that, under the previous assumptions, the density of the processes $(X_t)_{t\ge 0}, (X_t^{(\varepsilon)})_{t\ge 0}$ exists in positive time and satisfies some Gaussian bounds, see e.g Aronson \cite{aron:59}, Sheu \cite{sheu:91} or \cite{dela:meno:10} for extensions to some degenerate cases.

The goal of this work is  to investigate how the closeness of $(b_\varepsilon,\sigma_\varepsilon) $ and $(b,\sigma) $ is reflected on the respective densities of the associated processes. Important applications can for instance be found in mathematical finance.
If the dynamics of \eqref{1} models the evolution of the (log-)price of a financial asset, it is often very useful to know how a perturbation of the volatility $\sigma $ impacts the density, and therefore the associated option prices (see also Corielli \textit{et al.} \cite{cori:fosc:pasc:10} or Benhamou \textit{et al.} \cite{benh:gobe:miri:10} for related problems).  

In the framework of parameter estimation it can be useful, having at hand estimators $(b_\varepsilon,\sigma_\varepsilon) $ of the true parameters $(b,\sigma) $ and some controls for the differences $|b-b_\varepsilon|,|\sigma-\sigma_\varepsilon| $ in a suitable sense, to quantify the difference $p_\varepsilon-p $ of the densities corresponding respectively to the dynamics with the estimated parameters and the one of the model.

Another important application includes the case of mollification by spatial convolution. This specific kind of perturbation is useful to investigate the error between the densities of  a non-degenerate diffusion of type \eqref{1} with H\"older coefficients (or with piecewise smooth bounded drift) and its Euler scheme. In this framework, some explicit convergence results can be found in \cite{kona:meno:15:2}.

More generally, this situation can appear in every applicative field for which the diffusion coefficient might be misspecified.

The previously mentioned Gaussian bounds on the density are derived through the so-called \textit{parametrix} expansion which will be the crux of our approach. Roughly speaking, it consists in approximating the process by a proxy which has a known density, here a Gaussian one, and then in investigating the difference through the Kolmogorov equations.  Various approaches to the parametrix expansion exist, see e.g. Il'in \textit{et al.}
\cite{ilin:kala:olei:62}, Friedman \cite{frie:64}
 and McKean and Singer \cite{mcke:sing:67}. The latter approach will be the one used in this work since it appears to be the most adapted to handle 
 coefficients with no \textit{a priori} smoothness in time and can also be directly extended to the discrete case for Markov chain approximations of equations \eqref{1} and \eqref{2}. Let us mention in this setting the works of Konakov \textit{et al.}, see \cite{kona:mamm:00}, 
 \cite{kona:mamm:02}. 

Our stability results will also apply to two Markov chains with respective dynamics: 
\begin{eqnarray}
Y_{t_{k+1}}&=&Y_{t_k}+b(t_k,Y_{t_k})h+\sigma(t_k,Y_{t_k})\sqrt h \xi_{k+1}, Y_{0}=x,\nonumber\\
Y_{t_{k+1}}^{(\varepsilon)}&=&Y_{t_k}^{(\varepsilon)}+b_\varepsilon(t_k,Y_{t_k}^{(\varepsilon)})h+\sigma_\varepsilon(t_k,Y_{t_k}^{(\varepsilon)})\sqrt h \xi_{k+1}, Y_{0}^{(\varepsilon)}=x,\label{DYN_M_P}
\end{eqnarray}
where $h>0$ is a given time step, for which we denote for all $k\ge 0,\ t_k:=kh $ and the $(\xi_k)_{k\ge 1} $ are centered i.i.d. random variables satisfying some integrability conditions. Again, the key tool will be the parametrix representation for the densities of the chains and the Gaussian local limit theorem.
\subsection{Assumptions and Main Results.}

Let us introduce the following assumptions. Below, the parameter $\varepsilon>0$ is fixed and the constants appearing in the assumptions \textbf{do not depend} on $\varepsilon $.
\begin{trivlist}
\item[(\textbf{A1)}] \textbf{(Boundedness of the coefficients)}. The components of the vector-valued functions $b(t,x), b_\varepsilon(t,x)$ and the
matrix-functions $\sigma(t,x),\sigma_\varepsilon(t,x)$ are bounded measurable. Specifically, there exist constants $K_1,K_2>0$  s.t.
\begin{eqnarray*}
\sup_{(t,x)\in [0,T]\times \R^d}|b(t,x)|+\sup_{(t,x)\in [0,T]\times \R^d}|b_\varepsilon(t,x)|\le K_1,\\
\sup_{(t,x)\in [0,T]\times \R^d}|\sigma(t,x)|+\sup_{(t,x)\in [0,T]\times \R^d}|\sigma_\varepsilon(t,x)|\le K_2.
\end{eqnarray*}
\item[(\textbf{A2})] \textbf{(Uniform Ellipticity)}. The matrices $a:=\sigma\sigma^*, a_\varepsilon:=\sigma_\varepsilon\sigma_\varepsilon^*$ are uniformly elliptic, i.e. there exists $\Lambda \ge 1,\ \forall (t,x,\xi)\in [0,T]\times (\R^d)^2$,
\begin{eqnarray*}
\Lambda^{-1} |\xi|^2 
\le \langle a(t,x)\xi,\xi\rangle \le \Lambda |\xi|^2,
\Lambda^{-1} |\xi|^2 
\le \langle a_\varepsilon(t,x)\xi,\xi\rangle \le \Lambda |\xi|^2.
\end{eqnarray*}
\item[(\textbf{A3})] \textbf{(H\"older continuity in space)}.  For some $ \gamma \in (0,1]$ , $\kappa<\infty $, for all $t\in [0,T]$,
\begin{eqnarray*}
\left\vert \sigma(t,x)-\sigma(t,y)\right\vert+\left\vert \sigma_{\varepsilon}(t,x)-\sigma_{\varepsilon}(t,y)\right \vert \leq 
\kappa\left\vert x-y\right\vert ^{\gamma }.
\end{eqnarray*}
Observe that the last condition also readily gives, thanks to the boundedness of $\sigma,\sigma_\varepsilon $ that $a,a_\varepsilon $ are also uniformly $\gamma $-H\"older continuous.\\
\end{trivlist}

For a given $\varepsilon>0 $, we say that assumption \textbf{(A)} holds when conditions \textbf{(A1)}-\textbf{(A3)} are in force. Let us now introduce, under \textbf{(A)}, the quantities that will bound the difference of the densities in our main results below. Set for $\varepsilon>0$:
\begin{eqnarray*}
\Delta _{\varepsilon,b,\infty}:= \sup_{(t,x)\in [0,T]\times \R^d} \left\vert
b(t,x)-b_{\varepsilon}(t,x)\right\vert ,\ \forall q\in (1,+\infty),\ \Delta _{\varepsilon,b,q}:=\sup_{t\in[0,T]}\|b(t,.)-b_\varepsilon(t,.)\|_{L^q(\R^d)} .
\end{eqnarray*}
Since $\sigma,\sigma_\varepsilon$ are both $\gamma $-H\"older continuous, see (\textbf{A3}) we also define 
\begin{equation}
\label{DEF_DELTA_EPS_SIG_GAMMA}
\Delta_{\varepsilon,\sigma,\gamma}:=\sup_{u\in [0,T]}|\sigma(u,.)-\sigma_{\varepsilon}(u,.)|_{\gamma},
\end{equation}
where for $\gamma\in (0,1] $, $|.|_\gamma $ stands for the usual H\"older norm in space on $C_b^\gamma(\R^d,\R^d\otimes \R^d) $ (space of H\"older continuous bounded functions, see e.g. Krylov \cite{kryl:96})  i.e. :
$$|f|_\gamma:=\sup_{x\in \R^d}|f(x)|+[f]_\gamma,\ [f]_\gamma:= \sup_{x\neq y,(x,y)\in (\R^d)^2}\frac{|f(x)-f(y)|}{|x-y|^\gamma}.$$

 The previous control in particular implies for all $(u,x,y)\in [0,T]\times (\R^d)^2 $:
$$|a(u,x)-a(u,y) - a_{\varepsilon}(u,x) + a_{\varepsilon}(u,y)| \le 2(K_2+\kappa)\Delta_{\varepsilon,\sigma,\gamma} |x-y|^{\gamma}.$$
We eventually set for $q\in (1,+\infty]$,
\begin{equation}
\label{DEF_D_EPS}
\Delta_{\varepsilon,\gamma,q}:=\Delta_{\varepsilon,\sigma,\gamma}+\Delta_{\varepsilon,b,q},
\end{equation}
which will be the key quantity governing the error in our results.

We will denote, from now on, by $C$ a constant depending on the parameters appearing in \textbf{(A)} and $T$. We reserve the notation $c$ for constants that only depend on \textbf{(A)} but not on $T$. The values of $C,c$ may change from line to line and \textbf{do not depend on the considered parameter $\varepsilon $}. Also, for given integers $i,j\in \N$ s.t. $i<j $, we will denote by $\leftB i,j\rightB $ the set $\{i,i+1,\cdots, j\} $.

We are now in position to state our main results.

\begin{THM}[Stability Control for diffusions]
\label{MTHM}
Fix $\varepsilon>0 $ and a final deterministic time horizon $T>0$. Under (\textbf{A}) and for $q>d$, there exist 
constants 
$C:=C(q)\ge 1$ and $c\in (0,1] $ s.t. for all $0\le s<t\le T, (x,y)\in (\R^d)^2$:
\begin{equation}
\label{CTR_D}
p_c(t-s,y-x)^{-1}|(p-p_\varepsilon)(s,t,x,y)|\le C \Delta_{\varepsilon,\gamma,q} ,
\end{equation}
where $p(s,t,x,.), p_\varepsilon(s,t,x,.) $ respectively stand for the transition densities  at time $t$ of equations \eqref{1}, \eqref{2} starting from $ x$ at time $s$.
Also, we denote for a given $c>0$ and for all $u>0,z \in  \R^d$, $p_c(u,z):=\frac{c^{d/2}}{(2\pi u)^{d/2}} \exp(-c\frac{|z|^2}{2u})$. 
\end{THM}
\begin{REM}[About the constants]\label{RQ_CONST}
We mention that the constant $C:=C(q)$ in \eqref{CTR_D} explodes when $q\downarrow d $ and is decreasing in $q$ (see Lemmas \ref{ONE_STEP_C} and \ref{LEMME_DIFF_ITE} below for details). In particular, it can be chosen uniformly as soon as $q\ge q_0>d$.   
\end{REM}
Before stating our results for Markov Chains we introduce two kinds of innovations in \eqref{DYN_M_P}. Namely:
\begin{trivlist}
\item[\A{I${}_G$}] The i.i.d. random variables $(\xi_k)_{k\ge 1} $ are Gaussian, with law ${\mathcal N}(0,I_{d}) $. In that case the dynamics in \eqref{DYN_M_P} corresponds to the Euler discretization of equations \eqref{1} and \eqref{2}.
\item[\A{I${}_{P,M}$}]  For a given integer $M>2d+5+\gamma $, the innovations $(\xi_k)_{k\ge 1} $ 
are centered 
and have $C^5 $ density $f_\xi $ which has, together with its derivatives up to order $5$, at least polynomial decay of order $M$. Namely, for all $z\in \R^d$ and multi-index $\nu, |\nu|\le 5 $:
\begin{equation}
\label{CONC_INNO}
|D^\nu f_\xi(z)|\le CQ_M(z),
\end{equation}
where we denote for all $r>d,\ z\in \R^d ,\ \ Q_r(z):=c_r\frac{1}{(1+|z|)^r},\ \int_{\R^d}dz Q_r(z)=1$.
\end{trivlist}

\begin{THM}[Stability Control for Markov Chains]
\label{MTHM_M}
Fix $\varepsilon>0 $ and a final deterministic time horizon $T>0$. For $h=T/N,\ N\in \N^*:=\N\setminus\{0\}$, we set for $i\in \N,\ t_i:=ih $.
Under (\textbf{A}), assuming that either \A{I${}_G $} or \A{I${}_{P,M} $} holds, and for $q>d$
there exist $C:=C(q)\ge 1,c\in (0,1]$ s.t. for all $0\le t_i<t_j\le T, (x,y)\in (\R^d)^2$: 
\begin{equation}
\label{CTR_M_G}
\chi_c(t_j-t_i,y-x)^{-1}|(p^h-p_\varepsilon^h)(t_i,t_j,x,y)|\le C \Delta_{\varepsilon,\gamma,q} ,
\end{equation}
where $p^h(t_i,t_j,x,.), p_\varepsilon^h(t_i,t_j,x,.) $ respectively stand for the transition densities  at time $t_j$ of the Markov Chains $Y $ and $Y^{(\varepsilon)} $   in \eqref{DYN_M_P} starting from $ x$ at time $t_i$. Also:
\begin{trivlist}
\item[-] If \A{I${}_G $} holds:
$$\chi_c(t_j-t_i,y-x):=p_c(t_j-t_i,y-x),$$
with $p_c$ as in Theorem \ref{MTHM}.
\item[-] If \A{I${}_{P,M} $}  holds:
\begin{equation*}
\chi_c(t_j-t_i,y-x):=\frac{c^{d}}{(t_j-t_i)^{d/2}} Q_{M-(d+5+\gamma)}\left(\frac{|y-x|}{(t_j-t_i)^{1/2}/c}\right).
\end{equation*}
\end{trivlist}
The behavior of $C:=C(q)$ is similar to what is described in Remark \ref{RQ_CONST}.
\end{THM}

\subsection{On Some Related Applications.}

\subsubsection{Model Sensitivity for Option Prices.}
Assume for instance that the (log)-price of a financial asset is given by the dynamics in \eqref{1}. Under suitable assumptions the price of an option on that asset writes at time $t$ and when $X_t=x $ as $\E[f(\exp(X_T^{t,x}))] $ up to an additional discounting factor. In the previous expression $f$ is the pay-off function. For a rather large class of pay-offs, say measurable functions with polynomial growth, including irregular ones,  Theorem \ref{MTHM} allows to specifically quantify how a perturbation of the coefficients impacts the option prices. Precisely for a given $\varepsilon>0 $, under \A{A}:
\begin{eqnarray*}
|{\mathcal E}_\varepsilon(t,T,x,f)|:=|\E[f(\exp(X_T^{t,x}))]-\E[f(\exp(X_T^{(\varepsilon),t,x}))]|
\le  C \Delta_{\varepsilon,\gamma,q}  \int_{\R^d} f(\exp(y))p_c(T-t,x,y)dy.
\end{eqnarray*}
This previous control can be as well exploited to investigate perturbations of a model which provides some closed formulas, e.g. a perturbation of the Black and Scholes model that would include a stochastic volatility taking for instance $\sigma_\varepsilon(x)=\sigma+\varepsilon \psi(x) $ for some bounded $\gamma $-H\"older continuous function $\psi$ and $\varepsilon $ small enough. In that case, assuming that the drift is known and unperturbed, we have $\Delta_{\varepsilon,\gamma,\infty}=|\sigma_\varepsilon-\sigma|_\gamma=\varepsilon |\psi|_\gamma $.

In connection with this application, we can quote the work of Corielli \textit{et al.} \cite{cori:fosc:pasc:10} who give estimates on option prices through parametrix expansions truncating the series. Some of their results, see e.g. their Theorem 3.1, can be related to a perturbation analysis since they obtain an approximation of an option price for a local volatility model in terms of the Black-Scholes price and a correction term corresponding  to the first order term in the parametrix series. A more probabilistic approach to similar problems can be found in Benhamou \textit{et al.} \cite{benh:gobe:miri:10}. However, none of the indicated works indeed deals with the global perturbation analysis we perform here.

\subsubsection{Weak Error Analysis}
It is well known that if the coefficients $b,\sigma$ in \eqref{1} are smooth and $a$ satisfies the non-degeneracy condition \A{A2}, then the weak error on the densities for the approximation by the Euler scheme is well controlled. Precisely, for a given time step $h>0$, let us set for $i\in \N, t_i:=ih $. Introduce now, for a given starting time $t_i$, the Euler scheme $X_{t_i}^h=x,\ \forall j> i,\ X_{t_{j+1}}^h=X_{t_j}^h+b(t_j,X_{t_j}^h)h+\sigma(t_j,X_{t_j}^h) (W_{t_{j+1}}-W_{t_j})$ and denote by $p^h(t_i,t_j,x,.) $ its density at time $t_j$. The dynamics of the Euler scheme clearly enters the scheme \eqref{DYN_M_P}. It can be derived from Konakov and Mammen \cite{kona:mamm:02} (see also Bally and Talay \cite{ball:tala:96:2} for an extension to the hypoelliptic setting) that:
$$|p-p^h|(t_i,t_j,x,y)\le Ch p_c(t_j-t_i,y-x).$$

If the coefficients in \eqref{1} are not smooth, it is then possible to use a mollification procedure, 
taking for $x\in \R^d$,
 $b_\varepsilon(t,x):=b(t,.)\star \rho_{\varepsilon}(x),\sigma_\varepsilon(t,x):=\sigma(t,.)\star \rho_{\varepsilon}(x)$ with $\rho_\varepsilon(x):=\varepsilon^{-d}\rho(x/\varepsilon) $ and 
 $\rho\in C^\infty(\R^d,\R^+),\ \int_{\R^d}\rho(x)dx=1, \supp(\rho)\subset K $ for some compact set $K$ of $\R^d$. For the mollifying kernel  $\rho_\varepsilon $,  one then easily checks that for $\gamma $-H\"older continuous in space coefficients $b,\sigma$ there exists $C$ s.t.
 \begin{equation}
\begin{split}
 \sup_{t\in [0,T]}|b(t,.)-b_\varepsilon(t,.)|_\infty\le C
 \varepsilon^\gamma,\ \sup_{t\in [0,T]}|\sigma(t,.)-\sigma_\varepsilon(t,.)|_\eta\le C \varepsilon^{\gamma-\eta},\ \eta\in (0,\gamma).
\end{split}
\label{CTR_INDICATIF}
 \end{equation}
 The important aspect is that we lose a bit with respect to the sup norm when investigating the H\"older norm.
 We then have by Theorems \ref{MTHM} and \ref{MTHM_M} and their proof, that, for $\gamma $-H\"older continuous in space coefficients $b,\sigma$ and taking $q=\infty$, there exist $c,C$ s.t. for all $0\le s<t\le T, 0\le t_i<t_j\le T,\  (x,y)\in (\R^d)^2$:
 \begin{eqnarray*}
 |(p-p_\varepsilon)(s,t,x,y)|\le CC_\eta \varepsilon^{\gamma-\eta} p_c(t-s,y-x),\ 
 |(p^h-p_\varepsilon^h)(t_i,t_j,x,y)|\le CC_\eta \varepsilon^{\gamma-\eta}
 p_c(t_j-t_i,y-x),
 \end{eqnarray*}
where the constant $C_\eta$ explodes when $\eta$ tends to $0$.

To investigate the global weak error $(p-p^h)(t_i,t_j,x,y)=\{(p-p_\varepsilon)+(p_\varepsilon-p_\varepsilon^h) +(p_\varepsilon^h-p^h)\}(t_i,t_j,x,y)  $, it therefore remains to analyze the contribution $(p_\varepsilon-p_\varepsilon^h)(t_i,t_j,x,y) $. The results of \cite{kona:mamm:02} indeed apply but yield $|(p_\varepsilon-p_\varepsilon^h)(t_i,t_j,x,y)| \le C_ \varepsilon h p_c(t_j-t_i,y-x) $ where $C_\varepsilon$ is explosive when $\varepsilon $ goes to zero. The global error thus writes:
\begin{equation*}
|(p-p^h)(t_i,t_j,x,y)|\le C\{ C_\eta\varepsilon^{\gamma-\eta}+C_\varepsilon h\}p_c(t_j-t_i,y-x),
\end{equation*}
and a  balance is needed to derive a global error bound. This is precisely the analysis we perform in \cite{kona:meno:15:2}. In this work, we extend to densities (up to a slowly growing factor) the results previously obtained by Mikulevi{\v{c}}ius and Platen \cite{miku:plat:91} on the weak error, i.e. they showed
$|\E[f(X_T)-f(X_T^h)]|\le Ch^{\gamma/2} $ provided $f\in C_b^{2+\gamma}(\R^d,\R) $.
Precisely, we obtain through a suitable analysis of the constants $C_\eta, C_\varepsilon$, which respectively depend on behavior of the parametrix series and of the derivatives of the heat kernel with mollified coefficients, that $|p-p^h|(t_i,t_j,x,y)\le Ch^{\gamma/2-\psi(h)} p_c(t_j-t_i,y-x)$ for a function $\psi(h)$ going to 0 as $h\rightarrow 0$ (which is induced by the previous loss of $\eta $ in \eqref{CTR_INDICATIF}). In the quoted work, we also obtain some error bounds for piecewise smooth drifts having a countable set of discontinuities. This part explicitly requires the stability result of Theorems \ref{MTHM}, Theorems \ref{MTHM_M} for $q<+\infty$. The idea being that the difference between the piece-wise smooth drift and its smooth approximation (actually the mollification procedure is only required around the points of discontinuity), is well controlled in $L^q$ norm, $q<+\infty $. 

\subsubsection{Extension to some Kinetic Models}
The results of Theorems \ref{MTHM} and \ref{MTHM_M} should extend without additional difficulties to the case of degenerate diffusions of the form:
\begin{equation}
\label{EQ_DEG}
\begin{split}
dX_t^1&=b(t,X_t) dt+\sigma(t,X_t) dW_t,\\
dX_t^2&= X_t^1 dt,
\end{split}
\end{equation}
denoting $X_t=(X_t^1,X_t^2) $, under the same previous assumptions on $b,\sigma$ when we consider perturbations of the non-degenerate components, i.e. for a given $\varepsilon>0$,  $X_t^{(\varepsilon)}=(X_t^{1,(\varepsilon)},X_t^{2,(\varepsilon)}) $ where:
\begin{equation}
\label{EQ_DEG_PERT}
\begin{split}
dX_t^{1,(\varepsilon)}&=b_{\varepsilon}(t,X_t^{(\varepsilon)}) dt+\sigma_\varepsilon(t,X_t^{(\varepsilon)}) dW_t,\\
dX_t^{2,(\varepsilon)}&= X_t^{1,(\varepsilon)} dt.
\end{split}
\end{equation}
Indeed, under \textbf{(A)}, the required parametrix expansions of the densities associated with the solutions of equation \eqref{EQ_DEG}, \eqref{EQ_DEG_PERT} have been established in \cite{kona:meno:molc:10}.\\ 

\subsubsection{A posteriori Controls in Parameter Estimation.}
Let us consider to illustrate this application a parametrized family of diffusions of the form:
\begin{equation}
\label{PARAM_DIFF_STAT}
dX_t=b(t,X_t)dt+\sigma(\vartheta,t,X_t)dW_t,
\end{equation}
where $\vartheta \in \Theta \subset \R^d $, the coefficients $b,\sigma$ are smooth, bounded and the non-degeneracy condition \A{A2} holds. A natural \textit{practical} problem consists in estimating the \textit{true} parameter $\vartheta $ from an observed discrete sample $(X_{t_i^n})_{i\in \leftB 0,n\rightB} $ where the $\{(t_i)_{i\in \leftB 0,n\rightB} \} $ form a partition of the observation interval, i.e. if $T=1$, $0=t_0^n<t_1^n<\cdot<t_n^n=1$.

Introducing the contrast:
$$U^n(\vartheta):=\frac 1n\sum_{i=1}^n \Big[\log \big(\det(a(\vartheta,t_{i-1}^n,X_{t_{i-1}^n})) \big)+ \langle a^{-1}(\vartheta,t_{i-1}^n,X_{t_{i-1}^n})X_i^n,X_i^n\rangle \Big], \ \forall i\in \leftB 1,n\rightB, X_i^n:=\frac{X_{t_i^n}-X_{t_{i-1}^n}}{\sqrt{t_i^n-t_{i-1}^n}},$$
and denoting by $\hat \vartheta_n $ the corresponding minimizer, it was shown by Genon-Catalot and Jacod \cite{geno:jaco:93} that under $\P^{\vartheta} $, $\sqrt{n}(\hat \vartheta_n-\vartheta) $ converges in law towards a \textit{mixed normal variable} $S$ which is, conditionally to $\F_1:=\sigma\{ (X_s)_{s\in [0,1]} \} $, centered and Gaussian. For a precise expression of the covariance which depends on the whole path of $(X_t)_{t\in [0,1]}$ we refer to Theorem 3 and its proof in \cite{geno:jaco:93}.

This means that, when $n$ is large, conditionally to $\F_1 $, we have on a subset $ \bar \Omega \subset \Omega$ which has \textit{high probability}, that $| \hat \vartheta_n-\vartheta|\le \frac{C}{\sqrt n} $ for a certain threshold $C$. Setting $\varepsilon_n=n^{-1/2} $, $\sigma_{\varepsilon_n}(t,x):=\sigma(\hat \vartheta_n,t,x) $ and with a slight abuse of notation $\sigma(t,x):=\sigma(\vartheta,t,x) $, one gets that, on  $\bar \Omega $: 
\begin{eqnarray*}
|\sigma(t,x)-\sigma_{\varepsilon_n}(t,x)-(\sigma(t,y)-\sigma_{\varepsilon_n}(t,y))|\le |x-y|\wedge C n^{-1/2}\Rightarrow |\sigma-\sigma_{\varepsilon}|_\eta\le (Cn^{-1/2})^{1-\eta},\ \eta\in (0,1].
\end{eqnarray*}
We can then invoke our Theorem \ref{MTHM} to compare the densities of the diffusions with the estimated parameter and 
the exact one in \eqref{PARAM_DIFF_STAT}. 



The paper is organized as follows. We recall in Section \ref{SEC_PARAM_DENS} some basic facts about parametrix expansions for the densities of diffusions and Markov Chains. We then detail in Section \ref{SEC_STAB} how to perform a stability analysis of the parametrix expansions in order to derive the results of Theorems \ref{MTHM} and \ref{MTHM_M}. 


\mysection{Derivation of formal series expansion for densities}
\label{SEC_PARAM_DENS}
\subsection{Parametrix Representation of the Density for Diffusions}
In the following, for given $(s,x)\in \R^+\times \R^d $, we use the standard Markov notation $(X_t^{s,x})_{t\ge s}$
to denote the solution of \eqref{1} starting from $x$ at time $s$.

Assume that  $(X_t^{s,x})_{t\ge s}$ has for all $t>s$ a smooth density $p(s,t,x,.)$ (which is the case if additionally to \A{A} the coefficients are smooth see e.g. Friedman \cite{frie:64}). We would like to estimate this density at a given point $y\in \R^d$.
To this end, we introduce the following Gaussian inhomogeneous process with spatial variable frozen at $y$. For all $(s,x)\in [0,T]\times \R^d,\ t\ge s $ we set:
\begin{equation*}
\tilde X_t^{y}=x +\int_s^t \sigma(u,y) dW_u.
\end{equation*}
Its density $\tilde p^y $ readily satisfies the  Kolmogorov backward equation:
\begin{equation}
\label{KOLM_GEL}
\begin{cases}
\partial_u \tilde p^{y}(u,t,z,y)+ \tilde L_u^{y} \tilde p^{y} (u,t,z,y)=0,\ s\le u<t, z\in \R^d,\\
 \tilde p^{y}(u,t,.,y) \underset{u\uparrow t}{\rightarrow} \delta_y(.),
 \end{cases}
\end{equation}
where for all $\varphi \in C_0^2(\R^d,\R)$ (twice continuously differentiable functions with compact support) and $z\in \R^d $:
$$\tilde L_u^y \varphi(z)=\frac{1}{2}{\rm Tr}\left(\sigma \sigma^*(u,y) D_z^2 \varphi(z) \right),$$
stands for the generator of $\tilde X^y $ at time $u$.

On the other hand, since we have assumed the density of $X$ to be smooth, it must satisfy the Kolmogorov forward equation (see e.g. Dynkin \cite{dynk:65}). For a given starting point $x\in \R^d$ at time $s$,
\begin{eqnarray}
\label{KOLM}
\begin{cases}
\partial_u  p(s,u,x,z)-L_u^* p (s,u,x,z)=0,\ s< u\le t, z\in \R^d,\\
 p(s, u,x,.) \underset{u\downarrow s}{\rightarrow} \delta_x(.),
 \end{cases}
\end{eqnarray}
where $L_u^* $ stands for the \textit{formal} adjoint (which is again well defined if the coefficients in \eqref{1} are smooth) of the generator of \eqref{1} which for all $\varphi \in C_0^2(\R^d,\R), z\in \R^d $ writes:
$$L_u\varphi(z)=\frac{1}{2}{\rm Tr}\left(\sigma \sigma^*(u,z) D_z^2 \varphi(z) \right)+\langle b(u,z),D_z\varphi(z)\rangle.$$

Equations \eqref{KOLM_GEL}, \eqref{KOLM} yield the formal expansion below which is initially due  to McKean and Singer \cite{mcke:sing:67}.
\begin{eqnarray}
\label{EQ_PARAM}
(p-\tilde p^{y})(s,t,x,y)=\bint{s}^{t} du \partial_u \bint{\R^d}^{} dz p(s,u,x,z) \tilde p^{y}(u,t,z,y)\nonumber\\
=\bint{s}^{t} du \bint{\R^d}^{}dz \left( \partial_u p(s,u,x,z)\tilde p^{y}(u,t,z,y)+  p(s,u,x,z) \partial_u\tilde p^{y}(u,t,z,y) \right)\nonumber \\
=\bint{s}^{t} du \bint{\R^d}^{}dz \left( L_u^* p(s,u,x,z)\tilde p^{y}(u,t,z,y)-  p(s,u,x,z) \tilde L_u^{y}\tilde p^{y}(u,t,z,y) \right)\nonumber \\
=\bint{s}^{t} du \bint{\R^d}^{}dz  p(s,u,x,z)(L_u -\tilde L_u^{y})\tilde p^{y}(u,t,z,y), 
\end{eqnarray}
using the Dirac convergence for the first equality, equations \eqref{KOLM} and \eqref{KOLM_GEL}  for the third one. We eventually take the adjoint for the last equality. Note carefully that the differentiation under the integral is also here formal since we would need to justify that it can actually be performed using some growth properties of the density and its derivatives which we \textit{a priori} do not know.

Let us now introduce the notation  
$$f\otimes g (s,t,x,y)=\bint{s}^{t} du \bint{\R^d}^{}dz f(s,u,x,z) g(u,t,z,y) $$ 
for the time-space convolution and let us define $\tilde p(s,t,x,y):=\tilde p^{y}(s,t,x,y)$, that is in $\tilde p(s,t,x,y) $ we consider the density of the frozen process at the final point and observe it at \textit{that specific} point.  We now introduce the \textit{parametrix} kernel:
\begin{equation}
\label{PARAM_KER}
H(s,t,x,y):=(L_s-\tilde L_s)\tilde p(s,t,x,y):=(L_s-\tilde L_s^{y})\tilde p^y(s,t,x,y).
\end{equation}
With those notations equation \eqref{EQ_PARAM} rewrites:
 $$(p-\tilde p)(s,t,x,y)=p\otimes H (s,t,x,y).$$
From this expression, the idea then consists in iterating this procedure for $p(s,u,x,z) $
in \eqref{EQ_PARAM} introducing the density of a process with frozen characteristics in $z$ which is here the integration variable. This yields to iterated convolutions of the kernel and leads to the formal expansion:
 \begin{eqnarray}
 \label{DEV_SER}
p(s,t,x,y)=\bsum{r=0}^{\infty} \tilde p\otimes H^{(r)}(s,t,x,y),
\end{eqnarray}
where $\tilde p\otimes H^{(0)}=\tilde p, H^{(r)}=H\otimes H^{(r-1)},\ r\ge 1$. Obtaining estimates on $p$ from the formal expression \eqref{DEV_SER} requires to have good controls on the right-hand side. The remarkable property of this formal expansion is now that the right-hand-side of \eqref{DEV_SER} only involves controls on Gaussian densities which in particular will provide, associated with our assumption \A{A} a \textit{smoothing} in time property for the kernel $H$. 

\begin{PROP}
\label{PROP_PARAM_D}
Under the sole assumption \A{A},  for $t>s $, the density of $X_t^{x,s}$ solving \eqref{1}
exists and can be written as in \eqref{DEV_SER}. 
\end{PROP}

\begin{proof}The proof can already be derived from a sensitivity argument. We first introduce two parametrix series of the form \eqref{DEV_SER}. Namely,
\begin{equation}
p(s,t,x,y):=\widetilde{p}(s,t,x,y)+\sum_{r=1}^{\infty
} \tilde p\otimes H^{(r)}(s,t,x,y)
\label{3}
\end{equation}%
and
\begin{equation}
p_{\varepsilon}(s,t,x,y):=\widetilde{p}_{\varepsilon}(s,t,x,y)+\sum_{r=1}^{\infty }\tilde p_{\varepsilon}\otimes H_{\varepsilon}^{(r)}(s,t,x,y).\label{4}
\end{equation}
Let us point out that, at this stage, $p$ and $p_{\varepsilon}$ are defined as sums of series. The purpose is then to identify those sums with the densities of the processes $X_t^{s,x},X_t^{(\varepsilon),s,x} $ at point $y$.

%
The convergence of the series \eqref{3}  and \eqref{4} is in some sense \textit{standard} (see e.g. \cite{meno:10} or Friedman \cite{frie:64}) under \A{A}. We recall for the sake of completeness the key steps for \eqref{3}. 

From direct computations, there exist $c_1\ge 1, c\in (0,1]$ s.t. for all $T>0$ and all multi-index $\alpha, |\alpha|\le 8 $,
\begin{align}
\forall 0\le u<t\le T,\ (z,y)\in (\R^d)^2,\ \left|D_z^\alpha \widetilde{p}(u,t,z,y)\right|\leq \frac{c_1
}{(t-u)^{|\alpha|/2}} p_{c}(t-u,y-z),  \label{5}
\end{align}%
where 
$$ p_{c}(t-u,y-z)=\frac{c^{d/2}}{(2\pi(t-u))^{d/2}}\exp \left( -\frac{c}{2}  \frac{|y-z|^2}{t-u}\right),$$
stands for the usual Gaussian density in $\R^d$ with 0 mean and covariance $(t-u)c^{-1}I_d $. From \eqref{5}, the boundedness of the drift and the H\"older continuity in space  of the diffusion matrix we readily get that there exists $c_1\ge 1,\ c\in (0,1]$,
\begin{equation}
\left\vert H(u,t,z,y)\right\vert \leq \frac{c_1(1\vee T^{(1-\gamma)/2})}{%
(t-u)^{1-\gamma /2}}p_{c}(t-u,z-y).  \label{8}
\end{equation}%
Now the key point is that the control \eqref{8} yields an integrable singularity giving a smoothing effect in time once integrated in space in the time-space convolutions  appearing in \eqref{3} and \eqref{4}. It follows by induction that:

\begin{eqnarray}
 |\widetilde{p}\otimes H^{(r)}(s,t,x,y)| 
\leq  ((1\vee T^{(1-\gamma)/2}) c_1)^{r+1}\prod_{i=1}^r B(\frac{\gamma }{2},1+(i-1)\frac{\gamma }{2})
p_{c}(t-s,y-x) (t-s)^{\frac{r \gamma}{2}}\notag\\
=  \frac{((1\vee T^{(1-\gamma)/2})c_1)^{r+1}\left[ \Gamma (\frac{\gamma }{2})\right] ^{r}}{\Gamma
(1+r\frac{\gamma }{2})}p_{c}(t-s,y-x) (t-s)^{\frac{r \gamma}{2}},  \label{10}
\end{eqnarray}  
where for $a,b>0, B(a,b)=\int_0^1 t^{-1+a}(1-t)^{-1+b}dt $ stands for the $\beta $-function, and using as well the identity $B(a,b)=\frac{\Gamma(a)\Gamma(b)}{\Gamma(a+b)} $ for the last equality.
These bounds readily yield the convergence of the series as well as a Gaussian upper-bound. Namely
\begin{eqnarray}
\label{GaussianBound}
p(s,t,x,y)\le  c_1 \exp((1\vee T^{(1-\gamma)/2})c_1[(t-s)^{\gamma/2}]) p_c(t-s,y-x).
\end{eqnarray}

An important application of the \textit{stability} of the perturbation consists in considering coefficients $b_{\varepsilon}:=b\star \zeta_{\varepsilon},\sigma_\varepsilon:=\sigma\star \zeta_{\varepsilon} $
in \eqref{4}, where $\zeta_{\varepsilon}$ is a mollifier in time and space. For mollified coefficients which satisfy the non-degeneracy assumption \A{A2}, the existence and smoothness of the density of 
 $X^{(\varepsilon)} $ is a standard property (see e.g. Friedman \cite{frie:64}, \cite{frie:75}). It is also well known, see \cite{kona:mamm:00}, that this density $p_\varepsilon $ admits the series representation given in  \eqref{4}. 
 Observe now carefully that the previous Gaussian bounds also hold for $p_{\varepsilon} $ uniformly in $\varepsilon$ and independently
of the mollifying procedure. This therefore gives that
\begin{equation}
\label{UCV}
p_{\varepsilon}(s,t,x,y)\underset{\varepsilon \rightarrow 0}{\longrightarrow} p(s,t,x,y),
\end{equation}
boundedly 	and uniformly. Thus, for every continuous bounded function $f$ we derive from  
the bounded convergence theorem and \eqref{GaussianBound} that for all $0\le s<t\le T, \ x\in \R^d $:
\begin{eqnarray}
\label{CONV1}
\E[f(X_t^{(\varepsilon),s,x})]=\int_{\R^d}f(y)p_{\varepsilon}(s,t,x,y)dy \underset{\varepsilon\rightarrow 0}{\longrightarrow}\int_{\R^d}f(y)p(s,t,x,y)dy.
\end{eqnarray} 
Taking $f=1$ gives that $\int_{\R^d} p(s,t,x,y)dy=1$ and the uniform convergence in \eqref{UCV} gives that $p(s,t,x,.) $ is non negative. We therefore derive that $p(s,t,x,\cdot) $ is a probability density on $\R^d$.


On the other hand, under \A{A}, we can derive from Theorem 11.3.4 of \cite{stro:vara:79} that $(X_s^{\varepsilon})_{s\in [0,T]}\overset{({\rm law})}{\underset{\varepsilon \rightarrow 0}{\Longrightarrow}}  (X_s)_{s\in [0,T]}$. This gives that for any bounded continuous function $f$:
$$\E[f(X_t^{(\varepsilon),s,x})] \underset{\varepsilon\rightarrow 0}{\longrightarrow} \E[f(X_t^{s,x})].$$
This convergence and \eqref{CONV1} then yield that the random variable $X_t^{s,x} $ admits $p(s,t,x,\cdot) $ as density.

We can thus now conclude that the processes $X,X^{(\varepsilon)}$ in \eqref{1}, \eqref{2} have transition densities given by the sum of the series 
\eqref{3}, \eqref{4}. 

\end{proof} 

\subsection{Parametrix for Markov Chains}
\label{PARAM_MC}
One of the main advantages of the formal expansion in \eqref{DEV_SER} is that it has a direct discrete counterpart in the Markov  chain setting. Indeed, denote by $(Y_{t_j}^{t_i,x})_{j\ge i} $ the Markov chain with dynamics \eqref{DYN_M_P} starting from $x $ at time $t_i$. Observe first that if the innovations $(\xi_k)_{k\ge 1} $ have a density then so does the chain at time $t_k $.

Let us now introduce its generator at time $t_i$, i.e. for all $\varphi \in C_0^2(\R^d,\R),\ x\in \R^d $:
$$L_{t_i}^h\varphi(x):=h^{-1}\E[\varphi(Y_{t_{i+1}}^{t_i,x})-\varphi(x)].$$

In order to give a representation of the density of $p^h(t_i,t_j,x,y) $ of $Y_{t_j}^{t_i,x} $ at point $y$ for $j>i$, we introduce similarly to the continuous case, the Markov chain (or inhomogeneous random walk) with coefficients frozen in space at $y$. For given $(t_i,x)\in [0,T]\times \R^d,\ t_j\ge t_i $ we set:
$$\tilde Y_{t_j}^{t_i,x,y}:=x+h^{1/2}\bsum{k=i}^{j-1} \sigma(t_k,y) \xi_{k+1},$$
and denote its density by $\tilde p^{h,y}(t_i,t_j,x,.) $.
Its generator at time $t_i $ writes for all $\varphi \in C_0^2(\R^d,\R), x\in \R^d $:
$$\tilde L_{t_i}^{h,y}\varphi(x)= h^{-1}\E[\varphi(\tilde Y_{t_{i+1}}^{t_i,x,y})-\varphi(x)].$$

Using the notation $\tilde p^h(t_i,t_j,x,y):=\tilde p^{h,y}(t_i,t_j,x,y) $, we introduce now for $0\le i<j\le N $ the \textit{parametrix} kernel:
$$H^h(t_i,t_j,x,y):=(L_{t_i}^h-\tilde L_{t_i}^{h,y})\tilde p^h(t_i+h,t_j,x,y). $$
Analogously to Lemma 3.6 in \cite{kona:mamm:00}, which follows from a direct algebraic manipulation, we derive the following representation for the density  which can be viewed as the Markov chain analogue of Proposition \ref{PROP_PARAM_D}.
\begin{PROP}[Parametrix Expansion for the Markov Chain]
\label{PROP_PARAM_C}
Assume \A{A} is in force. Then, for $0\le t_i<t_j\le T $,
$$p^h(t_i,t_j,x,y)=\sum_{r=0}^{j-i}\tilde p^h \otimes_h H^{h,(r)}(t_i,t_j,x,y),$$
where the discrete time convolution type operator $\otimes_{h} $ is defined by
$$f\otimes_h g(t_i,t_j,x,y)=\sum_{k=0}^{j-i-1}h \int_{\R^d} f(t_i,t_{i+k},x,z) g(t_{i+k},t_j,z,y)dz. $$
Also $g\otimes_h H^{h,(0)}=g $ and for all $r\ge 1, H^{h,(r)}:=H^{h}\otimes_h H^{h,(r-1)} $ denotes the $r$-fold discrete convolution of the kernel $H^h$.

\end{PROP} 
\mysection{Stability of Parametrix Series.}
\label{SEC_STAB}

We will now investigate more specifically the sensitivity of the density w.r.t.  the coefficients through the difference of the series. For a given fixed parameter $\varepsilon $, under \textbf{(A)} the densities $p(s,t,x,.),p_\varepsilon(s,t,x,\cdot) $ at time $t$ of the processes in \eqref{1}, \eqref{2} starting from $x$ at time $s$ both admit a parametrix expansion of the previous type. 


\subsection{Stability for Diffusions: Proof of Theorem \ref{MTHM}}
\label{STAB_DIFF}
Let us consider the difference between two parametrix expansions:
\begin{align}
&|p(s,t,x,y)-p_\varepsilon(s,t,x,y)|  
=|\sum^{\infty}_{r=0} \tilde p\otimes H^{(r)}(s,t,x,y)-\sum^{\infty}_{r=0} \tilde p_\varepsilon\otimes H_\varepsilon^{(r)}(s,t,x,y)|  \notag \\
&\le | (\tilde p - \tilde p_\varepsilon)(s,t,x,y)| + |\sum^{\infty}_{r=1} \tilde p\otimes H^{(r)}(s,t,x,y)-\sum^{\infty}_{r=1} \tilde p_\varepsilon\otimes H_\varepsilon^{(r)}(s,t,x,y)|. \label{c}
\end{align}
The strategy to study the above difference, using some well known properties of the Gaussian kernels and their derivatives recalled in \eqref{5},  consists in  first studying the difference of the \textit{main} terms.


We have the following Lemma.

\begin{LEMME}[Difference of the first terms and their derivatives]
\label{LEM_MT}
Under \textbf{(A)}, there exist $c_1\ge 1,\ c\in (0,1]$ s.t. for all $0\le s<t,\  (x,y)\in (\R^d)^2 $ and all multi-index $\alpha,\ |\alpha|\le 4 $,  
$$|D_x^\alpha \tilde p(s, t,x,y)-D_x^\alpha\tilde p_\varepsilon( s,t,x,y)|\le \frac{c_1}{(t-s)^{|\alpha|/2}}  \Delta_{\varepsilon,\sigma,\gamma} p_{c}(t-s,y-x).$$ 
\end{LEMME}

\begin{proof}
Let us first consider $|\alpha|=0 $ and introduce some notations. 
Set:
\begin{eqnarray}
\label{DEF_SIGMA_B_n}
\Sigma(s,t,y):=\int_s^t a(u,y)du,\ \Sigma_\varepsilon(s,t,y):=\int_s^t a_\varepsilon(u,y)du.
\end{eqnarray}
Let us now identify the columns of the matrices $\Sigma(s,t,y),\Sigma_\varepsilon(s,t,y) $ with $d$-dimensional column vectors, i.e. for $\Sigma(s,t,y) $:
\begin{eqnarray*}
\Sigma(s,t,y)=\left(\begin{array}{c|c|c|c}
\Sigma^1& \Sigma^2& \cdots  & \Sigma^d
\end{array}
 \right)(s,t,y).
\end{eqnarray*}
We now rewrite $\tilde p(s,t,x,y) $ and $\tilde p_\varepsilon(s,t,x,y)  $ in terms of vectors $\Theta(s,t,y), \Theta_\varepsilon (s,t,y)\in \R^{d^2}$:
\begin{eqnarray*}
\tilde p(s,t,x,y)=f_{x,y}(\Theta(s,t,y)), \ \Theta(s,t,y)=((\Sigma ^1)^*,\cdots,(\Sigma^d)^*)^*(s,t,y),\\
\tilde p_\varepsilon(s,t,x,y)=f_{x,y}(\Theta_\varepsilon (s,t,y)), \ \Theta_\varepsilon(s,t,y)=( (\Sigma_\varepsilon^1)^*,\cdots,(\Sigma_\varepsilon^d)^*)^*(s,t,y),
\end{eqnarray*}
where $(\cdot)^* $ stands for the transpose and
\begin{equation}
\begin{split}
f_{x,y}: \R^{d^2}& \rightarrow \R\\
  \Gamma&\mapsto  f_{x,y}(\Gamma)=\frac{1}{(2\pi)^{d/2}\det(\Gamma^{1:d})^{1/2}}
  \exp\left(-\frac 12\langle  (\Gamma^{1:d})^{-1}(y-x) ,y-x\rangle \right),  
\end{split} 
  \label{DEF_F}
\end{equation}
where $\Gamma:=\left(\begin{array}{c}\Gamma^1\\
\Gamma^2\\
\vdots\\
\Gamma^{d}
\end{array}
\right) $ and each  $(\Gamma^i)_{i\in\leftB 1,d\rightB} $  belongs to $\R^d  $. Also, we have denoted: 
$$\Gamma^{1:d}:=\left( \begin{array}{c|c|c|c}
\Gamma^1& \Gamma^2& \cdots  & \Gamma^{d}
\end{array}
 \right) ,$$
 the $d\times d$ matrix formed with the entries $(\Gamma_i)_{i\in \leftB 1,d\rightB} $, each entry being viewed as a column.

The multidimensional Taylor expansion now gives: 
\begin{equation}
\begin{split}|
(\tilde p-\tilde p_{\varepsilon})(s,t,x,y)|=|f_{x,y}(\Theta(s,t,y))-f_{x,y}(\Theta_{\varepsilon}(s,t,y))|
\\
=\bigg|\sum_{|\nu|=1}D_\nu f_{x,y}(\Theta(s,t,y))\{(\Theta_{\varepsilon}-\Theta)(s,t,y)\}^\nu
\\
+2\sum_{|\nu|=2}\frac{\{(\Theta_{\varepsilon}-\Theta)(s,t,y)\}^\nu}{\nu!}\int_0^1 (1-\delta)D^\nu f_{x,y}([\Theta+\delta(\Theta_{\varepsilon}-\Theta)](s,t,y)) d\delta\bigg|,
\\
\end{split}
\label{P_F}
\end{equation}
where for a multi-index $\nu:=(\nu_1,\cdots,\nu_{d^2})\in \N^{d^2}$, we denote by $|\nu|:=\sum_{i=1}^{d^2} \nu_i$ the length of the multi-index,
 $\nu!=\prod_{i=1}^{d^2}\nu_i! $ and for $h\in \R^{d^2}$, $h^\nu:=\prod_{i=1}^{d^2}h_i^{\nu_i} $ (with the convention that $0^0=1 $).
With these notations, 
from \eqref{DEF_SIGMA_B_n}, \eqref{DEF_F}, \eqref{P_F}  we get:
\begin{align}
|f_{x,y}(\Theta(s,t,y))-f_{x,y}(\Theta_{\varepsilon}(s,t,y))| 
&\le c\bigg\{ \sum_{|\nu|=1}^{}  |D^{\nu} f_{x,y}(\Theta(s,t,y))| \Delta_{\varepsilon,\sigma,\gamma}(t-s)  \nonumber \\
&  + \Delta_{\varepsilon,\sigma,\gamma}^2  (t-s)^2  \max_{\delta \in [0,1]} \sum_{|\nu|=2}  \left|D^\nu f_{x,y}([\Theta+\delta(\Theta_{\varepsilon}-\Theta)](s,t,y)) \right|\bigg\},\label{reterm}
\end{align}
recalling that $\Delta_{\varepsilon,\sigma,\gamma} $ has previously been defined in \eqref{DEF_DELTA_EPS_SIG_GAMMA}\footnote{Actually only the sup norm term in $\Delta_{\varepsilon,\sigma,\gamma} $ could be considered for this part of the anlalysis.}.

Since $f_{x,y} $ in \eqref{DEF_F} is a Gaussian density in the parameters $x,y$, 
we recall from Cramer and Leadbetter \cite{cram:lead:04} (see eq. (2.10.3) therein), that for all $\Gamma \in \R^{d^2} $ and any multi index $\nu,\ |\nu|\le 2 $:
 $$ D^{\nu} f_{x,y}(\Gamma)=\frac 1{2^{|\nu|}} \left( \prod_{i=1}^{d^2} (\frac{\partial^2 }{\partial_{x_{\lfloor \frac {i -1}d\rfloor+1}} \partial_{x_{i-\lfloor \frac {i-1} d\rfloor d}}   })^{\nu_i} f_{x,y}(\Gamma)\right),$$
 where $\lfloor \cdot\rfloor $ stands for the integer part. 
Hence, taking from \eqref{reterm}, for all $\delta \in [0,1] $, $\Gamma_{\varepsilon,\delta}(s,t,y):= [\Theta+\delta(\Theta_{\varepsilon}-\Theta)](s,t,y)$ yields, thanks to the non-degeneracy conditions (see equation \eqref{5}):  
\begin{equation}
\label{CTR_DER_DEUX}
|D^{\nu} f_{x,y}(\Gamma_{\varepsilon,\delta}(s,t,y))|\le \frac{\bar c_1}{(t-s)^{|\nu|}}f_{x,y}\big( \bar c\Gamma_{\varepsilon,\delta}(s,t,y)\big)
\le  \frac{\bar c_1}{(t-s)^{|\nu|}}p_{\bar c}(t-s,y-x),
\end{equation}
for some $\bar c_1\ge 1, \bar c\in (0,1] $.

Thus, from \eqref{DEF_F}, \eqref{P_F}, equations \eqref{reterm} and \eqref{CTR_DER_DEUX} give:
\begin{align*}
| \tilde p(s,t,x,y)- \tilde p_{\varepsilon}(s,t,x,y) | 
\le \bar c_1\Delta_{\varepsilon,\sigma,\gamma} p_{\bar c}(t-s,y-x). 
\end{align*}
Up to a modification of $\bar c_1,\bar c$ or $c_1,c $ in \eqref{5} we can assume that the statement of the lemma and \eqref{5} hold with the same constants $c_1,c $. 
The bounds for the derivatives are established similarly using the controls of \eqref{5}. 
This concludes the proof.
\end{proof}
\begin{REM}
Observe from equation \eqref{P_F} that the previous Lemma still holds with $\Delta_{\varepsilon,\sigma,\gamma} $ replaced by $\Delta_{\varepsilon,\sigma,\infty}:=\sup_{t\in [0,T]}|\sigma(t,.)-\sigma_\varepsilon(t,.)|_\infty $. The H\"older norm is required to control the differences of the parametrix kernels.
\end{REM}

The previous lemma quantifies how close are the main parts of the expansions. To proceed we need to 
consider the difference between the one-step convolutions. Combining the estimates of Lemmas \ref{LEM_MT} and \ref{ONE_STEP_C} below will yield by induction the result stated in Theorem \ref{MTHM}.

\begin{LEMME}[Control of the one-step convolution]
\label{ONE_STEP_C}
 For all $0\le s<t\le T,\ (x,y)\in (\R^d)^2 $  and for $q\in (d,+\infty]$:
\begin{equation}
\begin{split}
| \tilde p \otimes H^{(1)}(s,t,x,y)-\tilde p_{\varepsilon} \otimes H^{(1)}_{\varepsilon}(s,t,x,y) |  \\
\le 
 c_1^2 p_c(t-s,y-x)\Big\{ 2(1\vee T^{(1-\gamma)/2})^2 [\Delta_{\varepsilon,\sigma,\gamma}+\I_{q=+\infty}\Delta_{\varepsilon,b,+\infty}] B(1, \frac{\gamma}{2})(t-s)^{\frac{\gamma}{2}}
 \\+ \I_{q\in (d,+\infty)}\Delta_{\varepsilon,b,q}B(\frac12+\alpha(q),\alpha(q))(t-s)^{\alpha(q)}\Big\}
 \label{esti2},
 \end{split}
\end{equation}
where $c_1,c$ are as in Lemma \ref{LEM_MT} and for $q\in (d,+\infty] $ we set $\alpha(q)=\frac 12( 1- \frac dq) $. The above control then yields for a fixed $q\in (d,+\infty]$:
\begin{equation}
\label{THE_EQ_38}
\begin{split}
| \tilde p \otimes H^{(1)}(s,t,x,y)-\tilde p_{\varepsilon} \otimes H^{(1)}_{\varepsilon}(s,t,x,y) |  \\
\le 2 \bar C^2\Delta_{\varepsilon,\gamma,q}p_c(t-s,y-x) (t-s)^{\frac \gamma2\wedge \alpha( q)} (B(1,\frac \gamma 2)\vee B(\frac12+\alpha(q),\alpha(q))),\ \bar C=c_1(1\vee T^{(1-\gamma)/2}) ,
\end{split}
\end{equation}
which will be useful for the iteration (see Lemma \ref{LEMME_DIFF_ITE}).
\end{LEMME}
\begin{proof}
Let us write:
\begin{align}
&| \tilde p \otimes H^{(1)}(s,t,x,y)-\tilde p_{\varepsilon} \otimes H^{(1)}_{\varepsilon}(s,t,x,y) | \le 
|\tilde p- \tilde p_{\varepsilon}|\otimes |H| (s,t,x,y)  +\tilde p_{\varepsilon}\otimes | H- H_{\varepsilon}| (s,t,x,y)=:I+II .
\label{sum1}
\end{align}
From Lemma \ref{LEM_MT} and \eqref{8} we readily get for all $q\in (d,+\infty]$:
\begin{equation}
\label{CTR_I}|\tilde p- \tilde p_{\varepsilon}| \otimes |H|(s,t,x,y) |\le  ((1\vee T^{(1-\gamma)/2})c_1)^2 \Delta_{\varepsilon,\gamma,q} p_c(t-s,y-x)  B(1, \frac{\gamma}{2})(t-s)^{\frac{\gamma}{2}}.
\end{equation}
Now we will establish that for all $0\le u<t\le T, \ (z,y)\in (\R^d)^2 $ and $q=+\infty$:
\begin{equation}
\label{DIFF_H_HN}
|(H-H_{\varepsilon})(u,t,z,y)|\le \Delta_{\varepsilon,\gamma,\infty} \frac{(1\vee T^{(1-\gamma)/2})c_1}{(t-u)^{1-\frac \gamma 2}} p_c(t-u,y-z).
\end{equation}
Equations \eqref{DIFF_H_HN} and \eqref{5} give that $II$ can be handled as $I$ which yields the result for $q=+\infty $. It therefore remains to prove \eqref{DIFF_H_HN}.
Let us write with the notations of \eqref{DEF_F}:
\begin{eqnarray*}
(H-H_{\varepsilon})(u,t,z,y):=\bigg[\frac{1}2\Tr\bigg( (a(u,z)-a(u,y)) D_z^2 f_{z,y}\big(\Theta(u,t,y) \big)\bigg)
+\langle b(u,z), D_z f_{z,y}\big(\Theta(u,t,y) \big)\rangle \bigg]\\
-\bigg[\frac{1}2\Tr\bigg( (a_{\varepsilon}(u,z)-a_{\varepsilon}(u,y)) D_z^2 f_{z,y}\big(\Theta_{\varepsilon}(u,t,y) \big)\bigg)
+\langle b_{\varepsilon}(u,z), D_z f_{z,y}\big(\Theta_{\varepsilon}(u,t,y) \big)\rangle \bigg].
\end{eqnarray*}
Thus,
\begin{eqnarray}
(H-H_{\varepsilon})(u,t,z,y)
=\frac{1}2\bigg[\Tr\bigg( (a(u,z)-a(u,y)) \{ D_z^2 f_{z,y}\big(\Theta(u,t,y) \big)-D_z^2 f_{z,y}\big(\Theta_{\varepsilon}(u,t,y) \big)\}\bigg)\nonumber\\
-\Tr\bigg( [(a_{\varepsilon}(u,z)-a_{\varepsilon}(u,y)-(a(u,z)-a(u,y))] D_z^2 f_{z,y}\big(\Theta_{\varepsilon}(u,t,y) \big)\bigg)\bigg]\nonumber\\
+\bigg[\langle b(u,z), \{D_z f_{z,y}\big(\Theta(u,t,y) \big)-D_z f_{z,y}\big(\Theta_{\varepsilon}(u,t,y) \big) \}\rangle 
-\langle (b_{\varepsilon}(u,z)-b(u,z)) , D_z f_{z,y}\big(\Theta_{\varepsilon}(u,t,y) \big)\rangle
\bigg].\label{THE_DECOMP_DIFF_NOY}
\end{eqnarray}
Observe now that, similarly to \eqref{CTR_DER_DEUX} one has for all $ i\in \{1,2\}$:
\begin{eqnarray*}
 \ |D_z^i f_{z,y}\big(\Theta(u,t,y) \big)|+|D_z^i f_{z,y}\big(\Theta_{\varepsilon}(u,t,y) \big)|\le \frac{\tilde c_1 }{(t-u)^{i/2}}p_{\tilde c}(t-u,y-z),\\
|D_z^i f_{z,y}\big(\Theta(u,t,y) \big)-D_z^i f_{z,y}\big(\Theta_{\varepsilon}(u,t,y) \big)|\le \frac{\tilde c_1\Delta_{\varepsilon,\sigma,\gamma}}{(t-u)^{i/2}}p_{\tilde c}(t-u,y-z).
\end{eqnarray*}
Also, 
\begin{eqnarray*}
|(a_{\varepsilon}(u,z)-a_{\varepsilon}(u,y)-(a(u,z)-a(u,y))|\le c\Delta_{\varepsilon,\sigma,\gamma}|z-y|^{\gamma},\\
|b_\varepsilon(u,z)-b(u,z)|\le c\Delta_{\varepsilon,b,\infty}.
\end{eqnarray*}
Thus, provided that $c_1,c$ have been chosen large and small enough respectively in Lemma \ref{LEM_MT}, the definition in \eqref{DEF_D_EPS} gives:
\begin{eqnarray*}
|(H-H_{\varepsilon})(u,t,z,y)|\le \frac{(1\vee T^{(1-\gamma)/2})c_1\Delta_{\varepsilon,\gamma,\infty} }{(t-u)^{1-\gamma/2}}p_c(t-u,y-z).
\end{eqnarray*}
This establishes \eqref{DIFF_H_HN} for $q=+\infty $. For $q\in (d,+\infty)$ we have to use H\"older's inequality in the time-space convolution involving the difference of the drifts (last term in \eqref{THE_DECOMP_DIFF_NOY}). Set: 
$$D(s,t,x,y):=\int_s^t du \int_{\R^d}\tilde p_\varepsilon(s,u,x,z)  \langle (b_{\varepsilon}(u,z)-b(u,z)) , D_z f_{z,y}\big(\Theta_{\varepsilon}(u,t,y)\big)\rangle dz.
 $$
Denoting by $\bar q $ the conjugate of $q$, i.e. $q,\bar q>1, q^{-1}+{\bar q}^{-1}=1 $, we get from \eqref{5} and for $q>d$ that:
\begin{eqnarray*}
|D(s,t,x,y)|\le c_1^2 \int_s^t \frac{du}{(t-u)^{1/2}} \|b(u,.)-b_\varepsilon(u,.)\|_{L^q(\R^d)} \Big\{ \int_{\R^d}[p_c(u-s,z-x)  p_c(t-u,y-z)]^{\bar q} dz\Big\}^{1/\bar q} \\
\le c_1^2 \Delta_{\varepsilon,b,q}\int_s^t  \frac{c^d}{(2\pi)^{d(1-\frac{1}{\bar q})}(c\bar q)^{d/\bar q}}\Big\{ \int_{\R^d} p_{c\bar q}(u-s,z-x)p_{c\bar q}(t-u,y-z) dz \Big\}^{1/\bar q}\frac{du}{(u-s)^{\frac{d}{2}(1-\frac{1}{\bar q})}(t-u)^{\frac12+\frac{d}{2}(1-\frac{1}{\bar q})}}\\
\le c_1^2\left(\frac{c(t-s)}{2\pi}\right)^{\frac d2(1-\frac{1}{\bar q})}\bar q^{-\frac{d}{2 \bar q}}\Delta_{\varepsilon,b,q} p_c(t-s,y-x) \int_s^t \frac{du}{(u-s)^{\frac{d}{2}(1-\frac{1}{\bar q})}(t-u)^{\frac12+\frac{d}{2}(1-\frac{1}{\bar q})}}.
\end{eqnarray*}
Now, the constraint $d<q<+\infty$ precisely gives that $1<\bar q<d/(d-1)\Rightarrow \frac12+\frac d2(1-\frac{1}{\bar q})<1$ so that the last integral is well defined. We therefore derive:
\begin{eqnarray*}
|D(s,t,x,y)|\le c_1^2 (t-s)^{\frac12-\frac d2(1-\frac 1{\bar q})}\Delta_{\varepsilon,b,q}p_c(t-s,y-x)B(1-\frac d2(1-\frac{1}{\bar q}),\frac12-\frac d2(1-\frac{1}{\bar q})).
\end{eqnarray*}
In the case $d<q<+\infty $, recalling that $\alpha(q)=\frac 12(1-\frac dq) $, we eventually get :
\begin{equation}
\label{CTR_DIFF_FINAL_Q}
\begin{split}
 \tilde p_{\varepsilon}\otimes |H- H_{\varepsilon}| (s,t,x,y)\le  c_1^2p_c(t-s,y-x)\{ \Delta_{\varepsilon,b,q}(t-s)^{\alpha(q)}B(\frac12+\alpha ( q),\alpha ( q))\\
+ 2\Delta_{\varepsilon,\sigma,\gamma}(1\vee T^{(1-\gamma)/2})(t-s)^{\gamma/2}B(1,\gamma/2)\}.
\end{split}
\end{equation}

The statement now follows in whole generality from \eqref{sum1}, \eqref{CTR_I}, equations \eqref{DIFF_H_HN}, \eqref{5} for $q=\infty$ and \eqref{CTR_DIFF_FINAL_Q} for $d<q<+\infty $.
\end{proof}

The following Lemma associated with Lemmas \ref{LEM_MT} and \ref{ONE_STEP_C} allows to complete the proof of Theorem \ref{MTHM}.
\begin{LEMME}[Difference of the iterated kernels]
\label{LEMME_DIFF_ITE}
For all $0\le s<t\le T,\ (x,y)\in (\R^d)^2 $ and for all $q\in (d,+\infty]$, $r\in \N$: 
\begin{equation}
 |(\tilde p \otimes H^{(r)}-\tilde p_{\varepsilon} \otimes H_{\varepsilon}^{(r)})(s,t,x,y)| 
\le (r+1)
\Delta_{\varepsilon,\gamma,q} \frac{\bar C^{r+1}\left[ \Gamma (\frac{\gamma }{2}\wedge \alpha(q))\right] ^{r}}{\Gamma
(1+r(\frac{\gamma }{2}\wedge \alpha(q) ))}p_{c}(t-s,y-x) (t-s)^{r(\frac{ \gamma}{2} \wedge \alpha(q))}.
\label{ITE}
\end{equation}
\end{LEMME}
where $c $ is as in Lemma \ref{LEM_MT} and $\bar C $ as in equation \eqref{THE_EQ_38} of Lemma \ref{ONE_STEP_C}.
\begin{proof}
Observe that Lemmas \ref{LEM_MT} and  \ref{ONE_STEP_C} respectively give \eqref{ITE} for $r=0 $ and $r=1$. Let us assume that it holds for a given $r\in \N^*$ and let us prove it for  $r+1$.

Let us denote for all $r\ge 1,\ \eta_r(s,t,x,y):= |(\tilde p \otimes H^{(r)}-\tilde p_{\varepsilon} \otimes H_{\varepsilon}^{(r)})(s,t,x,y)| $. Write
\begin{align*}
\eta_{r+1}(s,t,x,y) & \le | [\tilde p\otimes H^{(r)}- \tilde p_{\varepsilon}\otimes H_{\varepsilon}^{(r)} ]\otimes H (s,t,x,y)|+|\tilde p_{\varepsilon}\otimes H_{\varepsilon}^{(r)} \otimes (H-H_{\varepsilon})(s,t,x,y) |\\
&\le \eta_r \otimes |H| (s,t,x,y)+ |\tilde p_{\varepsilon}\otimes H_{\varepsilon}^{(r)}| \otimes |(H-H_{\varepsilon})|(s,t,x,y).
\end{align*}
Recall now that under \A{A}, the terms $|H|(s,t,x,y)$ and $|\tilde p_{\varepsilon}\otimes H_{\varepsilon}^{(r)}|$ satisfy respectively and uniformly in $\varepsilon$ the controls of equations
\eqref{8}, \eqref{10}. The result then follows from the proof of Lemma \ref{ONE_STEP_C} (see equation \eqref{DIFF_H_HN} for $q=\infty $ and \eqref{CTR_DIFF_FINAL_Q} for $q\in (d,+\infty) $) and the induction hypothesis.
\end{proof}

Theorem \ref{MTHM} now simply follows from the controls of Lemma \ref{LEMME_DIFF_ITE}, the parametrix expansions \eqref{3} and \eqref{4} of the densities $p,p_{\varepsilon} $ and the asymptotics of the gamma function.

\subsection{Stability for Markov Chains. }
In this Section we prove Theorem \ref{MTHM_M}. 
The strategy is rather similar to the one of Section \ref{STAB_DIFF} thanks to the series representation of the densities of the chains given in Proposition \ref{PROP_PARAM_C}.

Recall first from Section \ref{PARAM_MC} that we have the following representations for the density $p^h,p_{\varepsilon}^{h} $ of the Markov chains $Y,Y^{(\varepsilon)} $ in \eqref{DYN_M_P}. For all $0\le t_i<t_j\le T, \ (x,y)\in (\R^d)^2 $:
\begin{equation*}
\begin{split}
p^h(t_i,t_j,x,y)=\sum_{r=0}^{j-i} \tilde p^h\otimes_h H^{h,(r)}(t_i,t_j,x,y),\\ 
p_{\varepsilon}^h(t_i,t_j,x,y)=\sum_{r=0}^{j-i} \tilde p_{\varepsilon}^h\otimes_h H_{\varepsilon}^{h,(r)}(t_i,t_j,x,y).
\end{split}
\end{equation*}

\subsubsection{Comparison of the frozen densities}

The first key point for the analysis with Markov chains is the following Lemma.
\begin{LEMME}[Controls and Comparison of the densities and their derivatives]
\label{LEMME_COMP_MC}
There exist $c,c_1$ s.t. for all $0\le t_i<t_j\le T,\ (x,y)\in (\R^d)^2 $ and for all multi-index $ \alpha,\ |\alpha|\le 4$:
\begin{eqnarray*}
|D_x^\alpha \tilde p^h(t_i,t_j,x,y)|+|D_x^\alpha \tilde p_{\varepsilon}^h(t_i,t_j,x,y)|
\le  \frac{1}{(t_j-t_i)^{|\alpha|/2}}\psi_{c,c_1}(t_j-t_i,y-x),\\
|D_x^\alpha \tilde p^h(t_i,t_j,x,y)-D_x^\alpha \tilde p_{\varepsilon}^h(t_i,t_j,x,y)|
\le \frac{\Delta_{\varepsilon,\sigma,\gamma}}{(t_j-t_i)^{|\alpha|/2}} \psi_{c,c_1}(t_j-t_i,y-x),
\end{eqnarray*}
where 
\begin{trivlist}
\item[-] Under \A{I${}_G$}:
\begin{equation*}
\psi_{c,c_1}(t_j-t_i,y-x):=c_1 p_c(t_j-t_i,y-x),
\end{equation*}
\item[-] Under \A{I${}_{P,M}$}:
\begin{equation*}
\psi_{c,c_1}(t_j-t_i,y-x):=\frac{c_1 c^{d}
}{(t_j-t_i)^{d/2}} Q_{M-d-5}\left(\frac{|y-x|}{(t_j-t_i)^{1/2}/c}\right).
\end{equation*}
\end{trivlist}
\end{LEMME}
\begin{proof}
Note first that under \A{I${}_G$} the statement has already been proved in Lemma \ref{LEM_MT}. We thus assume that \A{I${}_{P,M}$} holds. 
Introduce first the random vectors with zero mean:
\begin{eqnarray*}
\tilde Z_{k,j}^{y}:=\frac{1}{(t_j-t_k)^{1/2}}\sum_{l=k}^{j-1} 
\sigma(t_l,y)\sqrt h \xi_{l+1},\ 
\tilde Z_{k,j}^{y,(\varepsilon)}:=\frac{1}{(t_j-t_k)^{1/2}}\sum_{l=k}^{j-1} 
\sigma_{\varepsilon}(t_l,y)\sqrt h \xi_{l+1}.
\end{eqnarray*}
Denoting by $q_{j-k},q_{j-k,\varepsilon} $ their respective densities, one  has:
\begin{eqnarray}
D_x^\alpha \tilde p^h(t_k,t_j,x,y)&=&\frac{1}{(t_j-t_k)^{(d+|\alpha|)/2}}(-1)^{|\alpha|}D_z^\alpha q_{j-k}(z)|_{z=\frac{y-x }{(t_j-t_k)^{1/2}}},\notag\\
D_x^\alpha \tilde p_{\varepsilon}^h(t_k,t_j,x,y)&=&\frac{1}{(t_j-t_k)^{(d+|\alpha|)/2}}(-1)^{|\alpha|}D_z^\alpha q_{j-k,\varepsilon}(z)|_{z=\frac{y-x }{(t_j-t_k)^{1/2}}}.\notag\\\label{REL}
\end{eqnarray}
From the Edgeworth expansion of Theorem 19.3 in Bhattacharya and Rao \cite{bhat:rao:76}, for $ q_{j-k},q_{j-k,\varepsilon}$, one readily derives under \A{A}, for $|\alpha|=0 $ that there exists $c_1$ s.t. for all $0\le t_k <t_j\le T,\ (x,y)\in (\R^d)^2 $,
\begin{equation}
\label{CTR_DEN_CDM}
\tilde p^h(t_k,t_j,x,y)+\tilde p_{\varepsilon}^h(t_k,t_j,x,y)\le \frac{c_1}{(t_j-t_k)^{d/2}}\frac{1}{\left(1+ \frac{|x-y|}{(t_j-t_k)^{1/2}}\right)^{m}},
\end{equation}
for all integer $m<M-d $, where we recall that $M $ stands for the initial decay of the density $f_\xi$ of the innovations bounded by $Q_M$ (see equation \eqref{CONC_INNO}).

We can as well derive similarly to the proof of Theorem 19.3 in \cite{bhat:rao:76}, see also Lemma 3.7 in \cite{kona:mamm:00}, that for all $\alpha, |\alpha|\le 4 $:
\begin{equation}
|D_x^\alpha\tilde p^h(t_k,t_j,x,y)|+|D_x^\alpha \tilde p_{\varepsilon}^h(t_k,t_j,x,y)|\le \frac{c_1}{(t_j-t_k)^{(d+|\alpha|)/2}}\frac{1}{\left(1+ \frac{|x-y|}{(t_j-t_k)^{1/2}}\right)^{m}},
\label{CTR_DEN_DER_CDM}
\end{equation}
for all $m<M-d-4$.
Note indeed that differentiating in $D_x^\alpha $ the density and the terms of the Edgeworth expansion corresponds to a multiplication of the Fourier transforms involved by $\zeta^\alpha $, $\zeta $ standing for the Fourier variable. Hence, from our smoothness assumptions in \A{I${}_{P,M}$}, after obvious modifications, the estimates of Theorem 9.11 and Lemma 14.3 from \cite{bhat:rao:76} apply for these derivatives. With these bounds, one then simply has to copy the proof of Theorem 19.3. Roughly speaking, taking derivatives deteriorates the concentration of the initial control in \eqref{CTR_DEN_CDM} up to the derivation order. 
On the other hand, the bound in \eqref{CTR_DEN_CDM} is itself deteriorated w.r.t. the initial concentration condition in \eqref{CONC_INNO}. The key point is that the techniques of Theorem 19.3 in \cite{bhat:rao:76} actually provide concentration bounds for inhomogeneous sums of random variables with concentration as in \eqref{CONC_INNO} in terms of the moments of the innovations. To explain the bound in \eqref{CTR_DEN_CDM} let us observe that the $m^{{\rm th}} $ moment of $\xi $ is finite for $m<M-d$.

Equations \eqref{CTR_DEN_CDM} and \eqref{CTR_DEN_DER_CDM} give the first part of the lemma. Still from the proof of Theorem 19.3 in \cite{bhat:rao:76}, one gets, under \A{A}, that there exists $C>0$ s.t. for all multi-indexes $\bar \alpha,\ |\bar \alpha|\le 4, \bar \beta,\ |\bar \beta|\le m\le M-d-5$ for all $j>k $:
\begin{equation}
\label{INT_F_CAR}
\int_{\R^d} |\zeta^{\bar \alpha} | \Big\{ |D_\zeta^{\bar \beta} \hat q_{j-k}(\zeta)|+|D_\zeta^{\bar \beta} \hat q_{j-k,\varepsilon}(\zeta)|\Big\} d\zeta\le C,
\end{equation}
where $\hat q_{j-k}(\zeta), \hat q_{j-k,\varepsilon}(\zeta)$ stand for the respective characteristic functions of the random variables $\tilde Z_{k,j}^{y},\ \tilde Z_{k,j}^{y,(\varepsilon)}$ at point $\zeta $.

To investigate the quantity $|D_x^{\alpha} \tilde p^h(t_k,t_j,x,y) - D_x^{\alpha} \tilde p^h_{\varepsilon}(t_k,t_j,x,y)|$ thanks to \eqref{REL} define now for all $\alpha, \ |\alpha|\le 4, \beta,\ |\beta|\le m \le M-d-5 $:
\begin{align}
\forall z\in \R^d,\ \Theta_{j-k,\varepsilon}(z):= z^{\beta} D^{\alpha}_z \left(q_{j-k}(z)-q_{j-k,\varepsilon}(z) 
 \right), \notag\\
\forall \zeta\in \R^d,\ \widehat \Theta_{j-k,\varepsilon}(\zeta):=(-i)^{|\alpha|+|\beta|} D_\zeta^{\beta}  \left( \zeta^\alpha  \left\{ \hat q_{j-k}(\zeta)-\hat q_{j-k,\varepsilon}(\zeta) \right\}
\right).\label{FOURIER_T}
\end{align}

Let us now estimate the difference between the characteristic functions. From the Leibniz formula, we are led to investigate for all multi-indexes $\bar \beta, \bar \alpha,\ |\bar \beta|\le |\beta|$, $|\bar \alpha| \le |\alpha|$ quantities of the form:
\begin{align}
& (i^{\bar \beta})^{-1}\zeta^{\bar \alpha  }(D_\zeta^{\bar \beta} \hat q_{j-k}(\zeta)- 
D_\zeta^{\bar \beta} \hat q_{j-k,\varepsilon}(\zeta) ) \notag\\
&=  \zeta^{\bar \alpha  }  \E \left[ ( \tilde Z _{k,j}^{y} )^{\bar \beta} \exp [i \zeta \cdot \tilde Z _{k,j}^{y}] - ( \tilde Z _{k,j}^{y,(\varepsilon)} )^{\bar \beta} \exp [i \zeta\cdot \tilde Z _{k,j}^{y, (\varepsilon)}] \right] 
\notag.
\end{align}
Assume first that $j>{k+1} $. In that case,
set now $\tilde Z _{k,j,1}^{y}:=\tilde Z_{k,\lceil (j+k)/2\rceil }^{y},\tilde Z_{k,j,2}^{y}:=\tilde Z_{k,j}^{y} -\tilde Z _{k,j,1}^{y} $. Denoting similarly $\tilde Z _{k,j,1}^{y,(\varepsilon)}:=\tilde Z_{k,\lceil (j+k)/2\rceil }^{y,(\varepsilon)},\tilde Z_{k,j,2}^{y,(\varepsilon)}:=\tilde Z_{k,j}^{y,(\varepsilon)} -\tilde Z _{k,j,1}^{y,(\varepsilon)} $ for the perturbed process, we get:
\begin{align}
& (i^{\bar \beta})^{-1}\zeta^{\bar \alpha  }(D_\zeta^{\bar \beta} \hat q_{j-k}(\zeta)- 
D_\zeta^{\bar \beta} \hat q_{j-k,\varepsilon}(\zeta) )=  \notag \\
& \zeta^{\bar \alpha }\Bigg\{\E \left[ ( \tilde Z _{k,j,1}^{y}  +  \tilde Z _{k,j,2}^{y})^{\bar \beta} \exp [i \zeta \cdot \tilde Z _{k,j,1}^{y}]\exp [i \zeta \cdot \tilde Z _{k,j,2}^{y}] \right] - 
\E \left[ ( \tilde Z _{k,j,1}^{y,(\varepsilon)}  +  \tilde Z _{k,j,2}^{y,(\varepsilon)})^{\bar \beta} \exp [i \zeta \cdot \tilde Z _{k,j,1}^{y,(\varepsilon)}]\exp [i \zeta \cdot \tilde  Z _{k,j,2}^{y,(\varepsilon)}] \right]\Bigg\}=  \notag  \\
& \zeta^{\bar \alpha}\Bigg\{ \sum_{l, |l|\le |\bar \beta|}^{} C_{\bar \beta}^l \E \Big[ (\tilde Z _{k,j,1}^{y} )^l  \exp [i \zeta\cdot \tilde Z _{k,j,1}^{y}]\Big]
\E \Big[ (\tilde Z _{k,j,2}^{y} )^{\bar \beta - l}  \exp [i \zeta \cdot \tilde Z _{k,j,2}^{y}] \Big]- \notag  \\
&  \sum_{l, |l|\le |\bar \beta|}^{} C_{\bar \beta}^l \E \Big[ (\tilde Z _{k,j,1}^{y,(\varepsilon)} )^l  \exp [i \zeta \cdot \tilde Z _{k,j,1}^{y,(\varepsilon)}]\Big]
\E\Big[  (\tilde Z _{k,j,2}^{y,(\varepsilon)} )^{\bar \beta - l}  \exp [i \zeta \cdot \tilde Z _{k,j,2}^{y,(\varepsilon)}]\Big] \Bigg\}= \notag  \\
&  \zeta^{\bar \alpha }\Bigg\{ \sum_{l, |l|\le |\bar \beta|}^{} C_{\bar \beta}^l \Bigl\{ \left[ \E  \Big[(\tilde Z _{k,j,1}^{y} )^l  \exp [i \zeta\cdot \tilde Z _{k,j,1}^{y}]\Big] -  \E\Big[  (\tilde Z _{k,j,1}^{y,(\varepsilon)} )^l  \exp [i \zeta \cdot \tilde Z _{k,j,1}^{y,(\varepsilon)}] \Big] \right]
\E  \Big[(\tilde Z _{k,j,2}^{y} )^{\bar \beta - l}  \exp [i \zeta\cdot \tilde Z _{k,j,2}^{y}]\Big] + \notag  \\ 
&
 \E \Big[ (\tilde Z _{k,j,1}^{y,(\varepsilon)} )^l  \exp [i \zeta \cdot \tilde Z _{k,j,1}^{y,(\varepsilon)}]  \Big]
\Big[ \E\Big[  (\tilde Z _{k,j,2}^{y} )^{\bar \beta - l}  \exp [i \zeta \cdot\tilde Z _{k,j,2}^{y}]\Big] -\E  \Big[(\tilde Z _{k,j,2}^{y,(\varepsilon)} )^{\bar \beta - l}  \exp [i \zeta \cdot \tilde Z _{k,j,2}^{y,(\varepsilon)} \Big]\Big]
\Bigr\} \Bigg\},\notag
\end{align}
where in the above expression we considered the binomial expansion for multi-indexes denoting by $C_{\bar \beta}^l:=\frac{\bar \beta !}{(\bar \beta-l)!l!} $ with the corresponding definitions for factorials (see the proof of Lemma \ref{LEM_MT}).
Introduce now, for a multi-index $ l, |l|\in \leftB 0,|\bar \beta|\rightB$, the functions:
\begin{eqnarray*}
\Psi_1^{\bar \alpha,\bar \beta-l}(\zeta)&:=&\zeta^{\bar \alpha}\E \Big[(\tilde Z _{k,j,2}^{y} )^{\bar \beta - l}  \exp [i \zeta\cdot \tilde Z _{k,j,2}^{y}] \Big],\ 
\Psi_2^{\bar \alpha,l}(\zeta):=\zeta^{\bar \alpha}\E \Big[ (\tilde Z _{k,j,1}^{y,(\varepsilon)} )^l  \exp [i \zeta \cdot \tilde Z _{k,j,1}^{y,(\varepsilon)}] \Big],
\end{eqnarray*}
and 
 \begin{eqnarray*}
  {\mathcal E}_{1,l}(\zeta):=\left[ \E \Big[ (\tilde Z _{k,j,1}^{y} )^l  \exp [i\zeta \cdot \tilde Z _{k,j,1}^{y}]\Big] -  \E\Big[  (\tilde Z _{k,j,1}^{y,(\varepsilon)} )^l  \exp [i \zeta \cdot \tilde Z _{k,j,1}^{y,(\varepsilon)}]\Big] \right],\\ 
  {\mathcal E}_{2,\bar \beta-l}(\zeta):=   \left[ \E  \Big[(\tilde Z _{k,j,2}^{y} )^{\bar \beta - l}  \exp [i \zeta \cdot \tilde Z _{k,j,2}^{y}]\Big] -\E\Big[  (\tilde Z _{k,j,2}^{y,(\varepsilon)} )^{\bar \beta - l}  \exp [i \zeta \cdot \tilde Z _{k,j,2}^{y,(\varepsilon)}]\Big] \right].
 \end{eqnarray*}   
Thus, we can rewrite from the previous computations:
\begin{equation}
\label{DECOMP_DIFF_DENS_CDM}
 (i^{\bar \beta})^{-1}\zeta^{\bar \alpha  }(D_\zeta^{\bar \beta} \hat q_{j-k}(\zeta)- 
D_\zeta^{\bar \beta} \hat q_{j-k,\varepsilon}(\zeta) )=
\sum_{l, |l|\le |\bar \beta|}^{} C_{\bar \beta}^l \left\{({\mathcal E}_{1,l}\Psi_1^{\bar \alpha,\bar \beta -l})(\zeta)+({\mathcal E}_{2,\bar \beta-l}\Psi_2^{\bar \alpha,l})(\zeta)\right\}.
\end{equation}
Recall from \eqref{INT_F_CAR} that we already have 
integrability for the contributions $\Psi_1^{\bar \alpha,\bar \beta-l}(\zeta)$
and 
$\Psi_2^{\bar \alpha, l}(\zeta)$. Let us thus proceed with the control of ${\mathcal E}_{1,l}(\zeta), {\mathcal E}_{2,\bar \beta-l}(\zeta) $. We only give details for $  {\mathcal E}_{1,l}(\zeta)$,
the contribution $  {\mathcal E}_{2,\bar \beta- l} $ can be handled similarly. We also consider $|l|\ge 2 $, since the cases $|l|< 2 $ can be handled more directly. Write:
\begin{eqnarray*}
|{\mathcal E}_{1,l}(\zeta)|\le \\
 \E[|(\tilde Z _{k,j,1}^{y} )^l-(\tilde Z _{k,j,1}^{y,(\varepsilon)} )^l|]+\E[|(\tilde Z _{k,j,1}^{y,(\varepsilon)} )^l||\exp(i \zeta \cdot \tilde Z _{k,j,1}^{y})-\exp(i\zeta \cdot \tilde Z _{k,j,1}^{y,(\varepsilon)})|]\\
\le C\Big\{\E[|\tilde Z _{k,j,1}^{y}-\tilde Z _{k,j,1}^{y,(\varepsilon)}| (|\tilde Z _{k,j,1}^{y}|^{|l|-1} +|\tilde Z _{k,j,1}^{y,(\varepsilon)}|^{|l|-1})] 
+\E[|\tilde Z _{k,j,1}^{y,(\varepsilon)}|^{|l|}|\zeta| |\tilde Z _{k,j,1}^{y}-\tilde Z _{k,j,1}^{y,(\varepsilon)}|]\Big\}.
\end{eqnarray*}
Apply now H\"older's inequality with $p_1=|l|,q_1=|l|/(|l|-1)$ for the first term and $p_2=(|l|+1)/|l|,\ q_2=|l|+1$ for the second one so that all the contribution appear with the same power (in order to equilibrate the constraints concerning the intregrability conditions). One gets:
\begin{eqnarray}
|{\mathcal E}_{1,l}(\zeta)|\le \notag\\
C\Big\{\E[|\tilde Z _{k,j,1}^{y}-\tilde Z _{k,j,1}^{y,(\varepsilon)}|^{|l|}]^{1/|l|}\{\E[|\tilde Z _{k,j,1}^{y}|^{|l|}]^{(|l|-1)/|l|}+\E[|\tilde Z _{k,j,1}^{y,(\varepsilon)}|^{|l|}]^{(|l|-1)/|l|}\}+\notag\\
|\zeta| \E[|\tilde Z _{k,j,1}^{y,(\varepsilon)}|^{|l|+1}]^{|l|/(|l|+1)}\E[|\tilde Z _{k,j,1}^{y}-\tilde Z _{k,j,1}^{y,(\varepsilon)}|^{|l|+1}]^{1/(|l|+1)}\Big\}.
\label{DECOUP_HOLDER}
\end{eqnarray}
The point is now to prove, since we have assumed $m\le M-d-5 \iff m+1\le M-d-4   $, that there exists $c$ s.t. for all $r\le m+1 $,
\begin{equation}
\label{CTR_MOM}
\E[|\tilde Z _{k,j,1}^{y}-\tilde Z _{k,j,1}^{y,(\varepsilon)}|^{r}]^{1/r}\le c\Delta_{\varepsilon,\sigma,\gamma},\  \E[|\tilde Z _{k,j,1}^{y}|^r]^{1/r}+\E[|\tilde Z _{k,j,1}^{y,(\varepsilon)}|^{r}]^{1/r}\le c.
\end{equation}
Let us establish the point for the difference, the other bounds can be derived similarly. Define for all $i\in \leftB k,j\rightB,\ \tilde M_{i}:=\sqrt h \sum_{r=k}^{i-1} (\sigma-\sigma_{\varepsilon})(t_r,y)\xi_{r+1} $. The process $(\tilde M_i)_{i\in \leftB k,j\rightB} $ is a square integrable martingale (in discrete time, w.r.t. $\F_{i}:=\Sigma(\xi_r, r\le i )$, $\Sigma $-field generated by the innovation up to the current time). Its quadratic variation writes $[\tilde M]_i=h\sum_{r=k}^{i-1} |(\sigma-\sigma_{\varepsilon})(t_r,y)|^2 |\xi_{r+1}|^2 $ and the Burkholder-Davies-Gundy inequalities, see e.g. Shiryaev \cite{shir:96}, give for all $r\le M-d-4  $:
\begin{equation}
\label{PREAL_CTR_PETIT}
\E[\sup_{i\in \leftB k,j\rightB}|\tilde M_i|^r]\le c_r\E[[\tilde M]_{j}^{r/2}]= c_rh^{r/2}\E[(\sum_{i=k}^{j-1}|(\sigma-\sigma_{\varepsilon})(t_i,y)|^2|\xi_{i+1}|^2)^{r/2}].
\end{equation}
If $r=2 $ one readily gets:
\begin{eqnarray*}
\E[|\tilde Z _{k,j,1}^{y}-\tilde Z _{k,j,1}^{y,(\varepsilon)}|^2]\le \frac{c_2}{(t_j-t_k)} \E[\sup_{i\in \leftB k,j\rightB}|\tilde M_i|^2]\le \frac{c_2h}{(t_j-t_k)}\Delta_{\varepsilon,\sigma,\gamma}^2 \sum_{i=k}^{j-1}\E[|\xi_{i+1}|^2]\le \bar c_2 \Delta_{\varepsilon,\sigma,\gamma}^2.
\end{eqnarray*}
Let us thus assume $r>2$ and derive from \eqref{PREAL_CTR_PETIT}
\begin{eqnarray*}
\E[|\tilde Z _{k,j,1}^{y}-\tilde Z _{k,j,1}^{y,(\varepsilon)}|^r]\le \frac{c_r}{(t_j-t_k)^{r/2}} \E[\sup_{i\in \leftB k,j\rightB}|\tilde M_i|^r]\\
\le \frac{c_rh^{r/2}}{(t_j-t_k)^{r/2}}\E[(\sum_{i=k}^{j-1}|(\sigma-\sigma_{\varepsilon})(t_i,y)|^r |\xi_{i+1}|^r)(\sum_{i=k}^{j-1}1)^{r/2(1-2/r)}], 
\end{eqnarray*}
applying H\"older's inequality for the counting measure with $p=r/2, q=r/(r-2)$ for the last inequality. This finally gives:
\begin{eqnarray*}
\E[|\tilde Z _{k,j,1}^{y}-\tilde Z _{k,j,1}^{y,(\varepsilon)}|^r]\le \frac{c_rh^{r/2}}{(t_j-t_k)^{r/2}}(j-k)^{r/2-1}\Delta_{\varepsilon,\sigma,\gamma}^r \sum_{i=k}^{j-1}\E[|\xi_{i+1}|^{r}]\le \bar c_r \Delta_{\varepsilon,\sigma,\gamma}^r.
\end{eqnarray*}
Since we have assumed $r\le m+1\le M-d-4$, this gives the first control in \eqref{CTR_MOM}. The other one readily follows replacing $\sigma-\sigma_{\varepsilon} $ by $\sigma$ or $\sigma_{\varepsilon} $.

From equations \eqref{DECOUP_HOLDER}, \eqref{CTR_MOM} and similar controls for  ${\mathcal E}_{2,\bar \beta-l}(\zeta) $ we finally derive:
\begin{align*}
|{\mathcal E}_{1,l}(\zeta)|+|{\mathcal E}_{2,\bar \beta -l}(\zeta)|\le C_1 \Delta_{\varepsilon,\sigma,\gamma} (1+|\zeta|) .
\end{align*}
As a result we have from \eqref{FOURIER_T} and \eqref{DECOMP_DIFF_DENS_CDM}:
\begin{align*}
| D_\zeta^{ \beta}( \zeta^{ \alpha  } (\hat q_{j-k}(\zeta)-D_\zeta^{ \beta} \hat q_{j-k,\varepsilon}(\zeta)) )|\\
 \le C \Delta_{\varepsilon,\sigma,\gamma} \Bigg\{ \sum_{
{\tiny \begin{array}{c}
 \bar \beta, |\bar \beta|\le |\beta|\\
 \bar \alpha=\alpha-(\beta-\bar \beta).
 \end{array}}
 }
 \sum_{l, |l|\le |\bar \beta|}^{ } (|\Psi_1^{ \bar \alpha ,
 \bar \beta-l }(\zeta)|+|\Psi_2^{\bar \alpha
 ,
 l }(\zeta)|)(1+|\zeta|)\Bigg\}.
\end{align*}

We finally derive from \eqref{FOURIER_T} and \eqref{INT_F_CAR} (which thanks to the smoothness assumption on $Q_M$ in \A{I${}_{P,M}$} holds as well for a multi-index $\bar \alpha, |\bar \alpha|=5$):
\begin{equation}
\label{CTR_NO_DRIFT}
|\Theta_{j-k,\varepsilon}(z)|\le \frac{1}{(2\pi)^d} \int_{\R^d}|\hat \Theta_{j-k,\varepsilon}(\zeta)|d\zeta \le c\Delta_{\varepsilon,\sigma,\gamma}.
\end{equation}
From \eqref{REL} this concludes the proof for $j>k+1$. 
If $j={k+1} $ the previous arguments can be simplified and lead to the same results. 
\end{proof}


\subsubsection{Comparison of the parametrix kernels}
This step is crucial and actually the key to the result for the Markov chains. We focus for simplicity on the case $q=+\infty$, for which pointwise controls for the differences between the drift coefficients are available, and which already emphasizes all the difficulties. The case $q\in (d,+\infty)$ for the drifts could be handled  as in Lemma \ref{ONE_STEP_C}, using similar H\"older inequalities.

We actually have the following Lemma.
\begin{LEMME}[Control of the One-Step Convolution for the Chain.]\label{LEMMA_DIFF_KER_MC}
There exists $c_1,c $ s.t. 
for all $0\le t_k<t_j\le T, (z,y)\in (\R^d)^2 $:
$$|(H^h-H_{\varepsilon}^h)(t_k,t_j,z,y)|\le  \frac{\Delta_{\varepsilon,\gamma,\infty}}{(t_j-t_k)^{1-\gamma/2}}\Phi_{c,c_1}(t_j-t_k,z-y),$$
with
\begin{trivlist}
\item[-] $\Phi_{c,c_1}(t_j-t_k,z-y)=\psi_{c,c_1}(t_j-t_k,z-y) $ under \A{I${}_G $}. 
\item[-] $\Phi_{c,c_1}(t_j-t_k,z-y)=\psi_{c,c_1}(t_j-t_k,z-y) \left(1+\frac{|z-y|}{(t_j-t_k)^{1/2}}\right)^{\gamma}$, under  \A{I${}_{P,M} $},
\end{trivlist}
where $\psi_{c,c_1} $ is defined according to the assumptions on the innovations in Lemma \ref{LEMME_COMP_MC}.
\end{LEMME}
\begin{proof} 
The case $k=j+1 $ involves directly differences of densities and could be treated more directly than the case $k>j+1$.
We thus focus on the latter. Introduce for $k\in \leftB 0,N\rightB, (x,w)\in (\R^d)^2 $ the one step transitions:
\begin{equation}
\label{DEF_TRANS}
\begin{split}
T^h(t_k,x,w)&:=b(t_k,x)h+h^{1/2}\sigma(t_k,x)w,\ T_{\varepsilon}^h(t_k,x,w):=b_{\varepsilon}(t_k,x)h+h^{1/2}\sigma_{\varepsilon}(t_k,x)w,\\
T_0^h(t_k,x,w)&:=h^{1/2}\sigma(t_k,x)w,\ T_{\varepsilon,0}^h(t_k,x,w):=h^{1/2}\sigma_{\varepsilon}(t_k,x)w.
\end{split}
\end{equation}
From the definition of $H^h,H_{\varepsilon}^h $, recalling that $f_\xi $ stands for the density of the innovation, the difference of the kernels writes:
\begin{eqnarray}
(H^h-H_{\varepsilon}^h)(t_k,t_j,z,y)
&=&h^{-1}\int_{\R^d} dw f_\xi(w)\Bigg[\Big\{\tilde p^h(t_{k+1},t_j,z+T^h(t_k,z,w),y)-\tilde p^h(t_{k+1},t_j,z+T_0^h(t_k,y,w),y) \Big\}\notag\\
&&-\Big\{\tilde p_{\varepsilon}^h(t_{k+1},t_j,z+T_{\varepsilon}^h(t_k,z,w),y)-\tilde p_{\varepsilon}^h(t_{k+1},t_j,z+T_{0,\varepsilon}^h(t_k,y,w),y) \Big\}\Bigg].\label{DIFF_KER_H}
\end{eqnarray}
Let us now perform a Taylor expansion at order 2 with integral rest. 
To this end, let us first introduce for $\lambda\in [0,1] $ the mappings:
\begin{eqnarray}
\begin{array}{cccl}
\varphi_\lambda^h: &\R^d\times \R^d &\longrightarrow &\R\\
                            &(T_1,T_2)&\longmapsto &\Tr\Big(D_z^2 \tilde p^h(t_{k+1},t_j,z+\lambda T_1,y)[T_2T_2^*] \Big),
\end{array}\notag \\
\begin{array}{cccl}
\varphi_{\lambda,\varepsilon}^h: &\R^d\times \R^d &\longrightarrow &\R\\
                            &(T_1,T_2)&\longmapsto &\Tr\Big(D_z^2 \tilde p_{\varepsilon}^h(t_{k+1},t_j,z+\lambda T_1,y)[T_2T_2^*] \Big),
\end{array}
\label{DEF_PHI}
\end{eqnarray}
where $T_2$ is viewed as a column vector and $T_2^*$ denotes its transpose.
Recalling as well that  $\xi $ is centered we get:
\begin{eqnarray}
\Delta H^{h,\varepsilon}(t_k,t_j,z,y):=(H^h-H_{\varepsilon}^h)(t_k,t_j,z,y)\notag\\
=\Bigg[\Big\langle D_z \tilde p^h(t_{k+1},t_j,z,y), b(t_k,z) \Big\rangle
-\Big\langle D_z\tilde p_{\varepsilon}^h(t_{k+1},t_j,z,y), b_{\varepsilon}(t_k,z) \Big\rangle\Bigg]\notag\\
+h^{-1}\int_{\R^d} dw f_\xi(w)\int_0^1 d\lambda (1-\lambda)\notag\\
\times \Bigg[ \Big\{\varphi_\lambda^h(T^h(t_k,z,w),T^h(t_k,z,w))-\varphi_\lambda^h(T_0^h(t_k,y,w),T_0^h(t_k,y,w))\Big\}\notag\\
-\Big\{\varphi_{\lambda,\varepsilon}^h(T_{\varepsilon}^h(t_k,z,w),T_{\varepsilon}^h(t_k,z,w))-\varphi_{\lambda,\varepsilon}^h(T_{0,\varepsilon}^h(t_k,y,w),T_{0,\varepsilon}^h(t_k,y,w))\Big\}
\Bigg]\notag\\
=:(\Delta_1 H^{h,\varepsilon}+\Delta_2 H^{h,\varepsilon})(t_k,t_j,z,y), \label{DECOUP_DELTA_H}
\end{eqnarray}
where for $i\in \{1,2\} $, $\Delta_i H^{h,\varepsilon} $ is associated with the terms of order $i$.
The idea is now to make 
 $\Delta_{\varepsilon,\gamma,\infty} $ appear explicitly. 
The term $\Delta_1 H^{h,\varepsilon}$ is the easiest to handle. We can indeed readily write:
\begin{eqnarray*}
\Delta_1 H^{h,\varepsilon}(t_k,t_j,z,y)\\
=\Bigg[\Big\langle D_z \tilde p^h(t_{k+1},t_j,z,y), [b(t_k,z) -b_{\varepsilon}(t_k,z) ]\Big\rangle
-\Big\langle (D_z\tilde p_{\varepsilon}^h-D_z\tilde p^h)(t_{k+1},t_j,z,y), b_{\varepsilon}(t_k,z) \Big\rangle\Bigg].
\end{eqnarray*} 

From Assumption \A{A3}, equation \eqref{DEF_D_EPS} and Lemma \ref{LEMME_COMP_MC} we derive for $q=+\infty$:
\begin{equation}
\label{CTR_DELTA_H1}
|\Delta_1 H^{h,\varepsilon}(t_k,t_j,z,y)|\le  
\frac{C\Delta_{\varepsilon,\gamma,\infty} 
}{(t_j-t_k)^{1/2}}\psi_{c,c_1}(t_j-t_k,y-z).
\end{equation}
The term $\Delta_2 H^{h,\varepsilon} $ is trickier to handle. Define to this end:
\begin{eqnarray*}
\Delta \varphi_\lambda^{h,\varepsilon}(t_k,z,y,w):=
\Big\{\varphi_\lambda^h(T^h(t_k,z,w),T^h(t_k,z,w))-\varphi_\lambda^h(T_0^h(t_k,y,w),T_0^h(t_k,y,w))\Big\}\notag\\
-\Big\{\varphi_{\lambda,\varepsilon}^h(T_{\varepsilon}^h(t_k,z,w),T_{\varepsilon}^h(t_k,z,w))-\varphi_{\lambda,\varepsilon}^h(T_{0,\varepsilon}^h(t_k,y,w),T_{0,\varepsilon}^h(t_k,y,w))\Big\}.
\end{eqnarray*}
Let us then decompose:
\begin{eqnarray}
\Delta \varphi_\lambda^{h,\varepsilon}(t_k,z,y,w)
:=\Bigg[\Big\{\varphi_\lambda^h(T^h(t_k,z,w),T^h(t_k,z,w))-\varphi_\lambda^h(T^h(t_k,z,w),T_0^h(t_k,y,w))\Big\}\notag\\
-\Big\{\varphi_{\lambda,\varepsilon}^h(T_{\varepsilon}^h(t_k,z,w),T_{\varepsilon}^h(t_k,z,w))-\varphi_{\lambda,\varepsilon}^h(T_{\varepsilon}^h(t_k,z,w),T_{0,\varepsilon}^h(t_k,y,w))\Big\}\Bigg]\notag\\
+\Bigg[\Big\{\varphi_\lambda^h(T^h(t_k,z,w),T_0^h(t_k,y,w))-\varphi_\lambda^h(T_0^h(t_k,y,w),T_0^h(t_k,y,w))\Big\}\notag\\
-\Big\{ \varphi_{\lambda,\varepsilon}^h(T_{\varepsilon}^h(t_k,y,w),T_{0,\varepsilon}^h(t_k,y,w))- \varphi_{\lambda,\varepsilon}^h(T_{0,\varepsilon}^h(t_k,z,w),T_{0,\varepsilon}^h(t_k,y,w))\Big\}
\Bigg]\notag\\
=:(\Delta_1 \varphi_\lambda^{h,\varepsilon}+\Delta_2 \varphi_\lambda^{h,\varepsilon})(t_k,z,y,w),\label{DEF_D_PHI_I}
\end{eqnarray}
and write from \eqref{DECOUP_DELTA_H}:
\begin{eqnarray}
\Delta_2 H^{h,\varepsilon}(t_k,t_j,z,y)=h^{-1}\int_{\R^d}dw f_\xi(w)\int_0^1 d\lambda(1-\lambda)(\Delta_1 \varphi_\lambda^{h,\varepsilon}+\Delta_2 \varphi_\lambda^{h,\varepsilon})(t_k,z,y,w)\notag\\
=:(\Delta_{21} H^{h,\varepsilon}+\Delta_{22} H^{h,\varepsilon})(t_k,t_j,z,y), \label{INT_DELTA_2_H}
\end{eqnarray}
for the associated contributions in $\Delta_2 H^{h,\varepsilon} $.
Again, we have to consider these two terms separately. 

\textbf{Term $\Delta_{21} H^{h,\varepsilon} $.} We first write from \eqref{DEF_D_PHI_I}:
\begin{eqnarray}
\Delta_1 \varphi_\lambda^{h,\varepsilon}(t_k,z,y,w)\notag\\
=\Bigg[\Big\{\varphi_\lambda^h(T^h(t_k,z,w),T^h(t_k,z,w))-\varphi_\lambda^h(T^h(t_k,z,w),T_0^h(t_k,y,w))\Big\}-\notag\\
\Big\{\varphi_\lambda^h(T^h(t_k,z,w),T_{\varepsilon}^h(t_k,z,w))-\varphi_\lambda^h(T^h(t_k,z,w),T_{0,\varepsilon}^h(t_k,y,w))\Big\}
\Bigg]\notag\\
+\Bigg[\Big\{\varphi_\lambda^h(T^h(t_k,z,w),T_{\varepsilon}^h(t_k,z,w))-\varphi_\lambda^h(T^h(t_k,z,w),T_{0,\varepsilon}^h(t_k,y,w))\Big\}\notag\\
-\Big\{\varphi_\lambda^h(T_{\varepsilon}^h(t_k,z,w),T_{\varepsilon}^h(t_k,z,w))-\varphi_\lambda^h(T_{\varepsilon}^h(t_k,z,w),T_{0,\varepsilon}^h(t_k,y,w))\Big\}\Bigg]\notag\\
-\Bigg[\Big\{\varphi_{\lambda,\varepsilon}^h(T_{\varepsilon}^h(t_k,z,w),T_{\varepsilon}^h(t_k,z,w))-\varphi_{\lambda,\varepsilon}^h(T_{\varepsilon}^h(t_k,z,w),T_{0,\varepsilon}^h(t_k,y,w))\Big\}\notag\\
-\Big\{\varphi_\lambda^h(T_{\varepsilon}^h(t_k,z,w),T_{\varepsilon}^h(t_k,z,w))-\varphi_\lambda^h(T_{\varepsilon}^h(t_k,z,w),T_{0,\varepsilon}^h(t_k,y,w))\Big\}\Bigg]\notag\\
=:\sum_{i=1}^3 \Delta_{1i} \varphi_\lambda^{h,\varepsilon}(t_k,z,y,w).
\label{SUM_D_1}
\end{eqnarray}
We now state some useful controls for the analysis. Namely, setting:
\begin{equation*}
\begin{split}
D(t_k,z,y,w)&:=T^h(t_k,z,w)T^h(t_k,z,w)^*-T_0^h(t_k,y,w)T_0^h(t_k,y,w)^*,\\
D_{\varepsilon}(t_k,z,y,w)&:=T_{\varepsilon}^h(t_k,z,w)T_{\varepsilon}^h(t_k,z,w)^*-T_{0,\varepsilon}^h(t_k,y,w)T_{0,\varepsilon}^h(t_k,y,w)^*,
\end{split}
\end{equation*}
we have from \A{A3} and equation \eqref{DEF_D_EPS} for $q=+\infty$ :
\begin{eqnarray}
(|D|+|D_{\varepsilon}|)(t_k,z,y,w)&\le& \bar c(h^2+h^{3/2}|w|+h(1\wedge |z-y|)^\gamma|w|^2),\notag\\
|D-D_{\varepsilon}|(t_k,z,y,w) &\le&  \bar c\Delta_{\varepsilon,\gamma,\infty}(h^2+h^{3/2}|w|+h(1\wedge |z-y|)^\gamma|w|^2).\label{CTR_DIFF}
\end{eqnarray}

From the definition of $\varphi_\lambda^h$ in \eqref{DEF_PHI}, equation \eqref{SUM_D_1}, the control \eqref{CTR_DIFF} and Lemma \ref{LEMME_COMP_MC}, we get:
\begin{equation}
\label{C_INTER_1}
\begin{split}
|\Delta_{11} \varphi_\lambda^{h,\varepsilon}|(t_k,z,y,w) \\
\le \bar c \Delta_{\varepsilon,\gamma,\infty} \frac{\psi_{c,c_1}(t_j-t_k,y-(z+\lambda T^h(t_k,z,w)))}{(t_j-t_k)}(h^2+h^{3/2}|w|+h(1\wedge|z-y|)^\gamma|w|^2).
\end{split}
\end{equation}

We would similarly get from Lemma \ref{LEMME_COMP_MC} and \eqref{CTR_DIFF}:
\begin{eqnarray}
|\Delta_{13} \varphi_\lambda^{h,\varepsilon}|(t_k,z,y,w)\notag\\
 \le \bar c \Delta_{\varepsilon,\gamma,\infty} \frac{\psi_{c,c_1}(t_j-t_k,y-(z+\lambda T_{\varepsilon}^h(t_k,z,w)))}{(t_j-t_k)}(h^2+h^{3/2}|w|+h(1\wedge|z-y|)^\gamma|w|^2),\notag\\
|\Delta_{12} \varphi_\lambda^{h,\varepsilon}|(t_k,z,y,w)\notag\\
\le   \frac{\psi_{c,c_1}(t_j-t_k,y-(z+\theta \lambda T^h(t_k,z,w)+(1-\theta)\lambda T_{\varepsilon}^h(t_k,z,w) ))}{(t_j-t_k)^{3/2}}\notag\\
\times |(T^h-T_{\varepsilon}^h)(t_k,z,w)||D_{\varepsilon}|(t_k,z,y,w)\notag\\
\le \bar c \Delta_{\varepsilon,\gamma,\infty} \frac{\psi_{c,c_1}(t_j-t_k,y-(z+\theta \lambda T^h(t_k,z,w)+(1-\theta)\lambda T_{\varepsilon}^h(t_k,z,w) ))}{(t_j-t_k)^{3/2}}\notag\\
\times  
(h^2+h^{3/2}|w|+h(1\wedge|z-y|)^\gamma|w|^2)(h+h^{1/2}|w|),
\label{C_INTER_2}
\end{eqnarray}
for some $\theta\in (0,1)$, using as well \eqref{DEF_TRANS} and \eqref{DEF_D_EPS} for the last inequality. The point is now to get rid of the transitions appearing in the function $\psi_{c,c_1} $. We separate here the two assumptions at hand.
\begin{trivlist}
\item[-] Under \A{I${}_G$}, it suffices to remark that by the convexity inequality $|z-y-\Theta|^2\ge \frac{1}{2}|z-y|^2-|\Theta|^2 $, for all $\Theta \in \R^d$:
\begin{eqnarray*}
\psi_{c,c_1}(t_j-t_k,y-z-\Theta)
\le c_1 
\frac{c^{d/2}}{(2\pi(t_j-t_k))^{d/2}}\exp\left(-\frac c4 \frac{|z-y|^2}{t_j-t_k}\right)\exp\left(\frac c2\frac{|\Theta|^2}{t_j-t_k}\right).
\end{eqnarray*}
Now, if $\Theta $ is one of the above transitions or linear combination of transitions, we get from \eqref{DEF_TRANS}:
\begin{eqnarray}
\label{C_INTER_G} \psi_{c,c_1}(t_j-t_k,y-z-\Theta)
\le c_1
\frac{(c/2)^{d/2}}{(2\pi(t_j-t_k))^{d/2}} \exp\left(-\frac c4 \frac{|z-y|^2}{t_j-t_k}\right)\exp(\frac c2 K_2^2 |w|^2),
\end{eqnarray}
up to a modification of $c_1$ observing that $h/(t_j-t_k)\le 1 $ and with $K_2 $ as in \A{A${}_1$}. Since $c$ can be chosen 
small enough in the previous controls, up to deteriorating the concentration properties in Lemma \ref{LEMME_COMP_MC}, the last term can be integrated by the standard Gaussian density $f_\xi $ appearing in \eqref{INT_DELTA_2_H}. We thus derive, from \eqref{C_INTER_G}, \eqref{C_INTER_1}, \eqref{C_INTER_2} and the definition in \eqref{SUM_D_1}, up to modifications of $c,c_1$:
\begin{equation*}
\begin{split}
|\Delta_1 \varphi_\lambda^{h,\varepsilon}|(t_k,z,y,w)\le\\
 \Delta_{\varepsilon,\gamma,\infty} h \bar c\psi_{c,c_1}(t_j-t_k,z-y)\exp(c|w|^2)\left\{ 1+ \frac{|w|}{(t_j-t_k)^{1/2}}+\frac{|z-y|^\gamma |w|^2}{t_j-t_k} \right\},
 \end{split}
\end{equation*}
which plugged into \eqref{INT_DELTA_2_H} yields up to modifications of $\bar c, c,c_1$:
\begin{equation}
\label{CTR_D_H21_HN}
|\Delta_{21} H^{h,\varepsilon}(t_k,t_j,z,y)|\le \bar c\frac{\Delta_{\varepsilon,\gamma,\infty} (1\vee T^{(1-\gamma)/2})\psi_{c,c_1}(t_j-t_k,z-y)}{(t_j-t_k)^{1-\gamma/2}}.
\end{equation}
\item[-] Under \A{I${}_{P,M}$}, we only detail the computations for the off diagonal regime $|z-y|\ge c(t_j-t_k)^{1/2} $ which is the most delicate to handle. In this case, we have to discuss according to the position of $w$ w.r.t. $y-z $. With the notations of \A{A2}, introduce
${\mathcal  D}:=\{\bar w\in \R^d: \{\Lambda h\}^{1/2} |\bar w|\le |z-y|/2 \} $.
If $w \in {\mathcal D}$, then, still from \eqref{C_INTER_1}, \eqref{C_INTER_2}, 
\begin{eqnarray*}
(|\Delta_{11} \varphi_\lambda^{h,\varepsilon}|+|\Delta_{13} \varphi_\lambda^{h,\varepsilon}|)(t_k,z,y,w)\\
\le  \bar c \Delta_{\varepsilon,,\gamma,\infty} \frac{\psi_{c,c_1}(t_j-t_k,y-z)}{(t_j-t_k)}(h^2+h^{3/2}|w|+h(1\wedge |z-y|)^\gamma|w|^2),\\
|\Delta_{12} \varphi_\lambda^{h,\varepsilon}|(t_k,z,y,w)\\
\le c \Delta_{\varepsilon,\gamma,\infty} \frac{\psi_{c,c_1}(t_j-t_k,y-z)}{(t_j-t_k)^{3/2}} 
(h^2+h^{3/2}|w|+h(1\wedge |z-y|)^\gamma|w|^2)(h+h^{1/2}|w|).
\end{eqnarray*}
On the other hand, when $w\not \in {\mathcal D}$ we use $f_\xi$ to make the off-diagonal bound of $ \psi_{c,c_1}(t_j-t_k,y-z)$ appear. Namely, we can write:
\begin{eqnarray}
f_\xi(w)&\le& c\frac{1}{(1+|w|)^M}\le c\frac{1}{(1+\frac{|z-y|}{h^{1/2}})^{M-(d+4)}}\frac{1}{(1+|w|)^{d+4}}\notag\\
&\le & c\frac{1}{(1+\frac{|z-y|}{(t_j-t_k)^{1/2}})^{M-(d+4)}}\frac{1}{(1+|w|)^{d+4}},\label{LOSS_CONC}
\end{eqnarray}
where the last splitting is performed in order to integrate the contribution in $|w|^3 $ coming from the upper bound for $|\Delta_{12}\varphi_\lambda^{h,\varepsilon}|$ in \eqref{C_INTER_2}. Plugging the above controls in \eqref{INT_DELTA_2_H} yields:
\begin{equation}
\label{CTR_D_H21_HN_INNO_P}|\Delta H_{21}^{h,\varepsilon}(t_k,t_j,z,y)|\le \frac{\Delta_{\varepsilon,\gamma,\infty} \Phi_{c,c_1}(t_j-t_k,z-y)}{(t_j-t_k)^{1-\gamma/2}}.
\end{equation}
We emphasize that in the case of innovations with polynomial decays, the control on the difference of the kernels again induces a loss of concentration of order $\gamma $ in order to equilibrate the time singularity. 
\end{trivlist}

\textbf{Term $\Delta_{22} H^{h,\varepsilon} $.} This term can be handled with the same arguments as $\Delta_{21} H^{h,\varepsilon} $. For the sake of completeness we anyhow specify how the different contributions appear.
Namely, with the notations of \eqref{DEF_D_PHI_I} and \eqref{INT_DELTA_2_H}:
\begin{eqnarray*}
\Delta_2 \varphi_\lambda^{h,\varepsilon}(t_k,z,y,w)=\\
\int_0^1 d\mu \Big\{\Big\langle D_{T_1}\varphi_{\lambda}^h(T_0^h(t_k,y,w)+\mu (T^h(t_k,z,w)-T_0^h(t_k,y,w)) ,T_0^h(t_k,y,w))
, T^h(t_k,z,w)-T_0^h(t_k,y,w) \Big\rangle\\
-\Big\langle D_{T_1}\varphi_{\lambda,\varepsilon}^h(T_{0,\varepsilon}^h(t_k,y,w)+\mu(T_{\varepsilon}^h(t_k,z,w)-T_{0,\varepsilon}^h(t_k,y,w)),T_{0,\varepsilon}^h(t_k,y,w))
, T_{\varepsilon}^h(t_k,z,w)-T_{0,\varepsilon}^h(t_k,y,w)\Big\rangle \Big\}\\
=\Big\{\int_0^1 d\mu \Big\{\Big\langle D_{T_1}\varphi_{\lambda}^h(T_0^h(t_k,y,w)+\mu (T^h(t_k,z,w)-T_0^h(t_k,y,w)) ,T_0^h(t_k,y,w)),\\
\big[(T^h(t_k,z,w)-T_0^h(t_k,y,w))-(T_{\varepsilon}^h(t_k,z,w)-T_{0,\varepsilon}^h(t_k,y,w))\big]\Big\rangle \Big\}\\
-\Big\{ \int_0^1 d\mu \Big[\Big\langle D_{T_1}\varphi_{\lambda,\varepsilon}^h(T_{0,\varepsilon}^h(t_k,y,w)+\mu(T_{\varepsilon}^h(t_k,z,w)-T_{0,\varepsilon}^h(t_k,y,w)),T_{0,\varepsilon}^h(t_k,y,w))\\
-D_{T_1}\varphi_{\lambda}^h(T_0^h(t_k,y,w)+\mu(T^h(t_k,z,w)-T_0^h(t_k,y,w)),T_0^h(t_k,y,w)) \Big],
 T_{\varepsilon}^h(t_k,z,w)-T_{0,\varepsilon}^h(t_k,y,w)\Big\rangle \Big\}\\
=:(\Delta_{21} \varphi_\lambda^{h,\varepsilon}+\Delta_{22} \varphi_\lambda^{h,\varepsilon})(t_k,z,y,w).
\end{eqnarray*}
In $\Delta_{21} \varphi_\lambda^{h,\varepsilon}$ we have sensitivities of order 3 for the density, giving time singularities in $(t_j-t_k)^{-3/2} $, which are again equilibrated by the the multiplicative factor:
\begin{eqnarray*}
|T_0^h(t_k,y,w)[T_0^h(t_k,y,w)]^*|\\
\times |(T^h(t_k,z,w)-T_0^h(t_k,y,w))-(T_{\varepsilon}^h(t_k,z,w)-T_{0,\varepsilon}^h(t_k,y,w))|\\
\le \bar c(h^2+h^{3/2}|w|+ h|w|^2)\Delta_{\varepsilon,\gamma,\infty}( h+h^{1/2}(1\wedge |z-y|)^\gamma|w|),
\end{eqnarray*}
where the last inequality is obtained similarly to \eqref{CTR_DIFF} using as well \eqref{DEF_D_EPS}. 
The same kind of controls can be established for $\Delta_{22} \varphi_\lambda^{h,\varepsilon}$. Anyhow, the analysis of this term leads to investigate the difference of third order derivatives, which finally yields contributions involving derivatives of order four. This is what induces the final concentration loss under \A{I${}_{P,M}$}, i.e. we need to integrate a term in $|w|^4 $ (see also equation \eqref{LOSS_CONC} in which we performed the splitting of $f_\xi $ on the \textit{off-diagonal} region to integrate a contribution in $|w|^3$).

We can thus claim that
\begin{equation*}
|\Delta H_{22}^{h,\varepsilon}(t_k,t_j,z,y)|\le \frac{\Delta_{\varepsilon,\gamma,\infty} \Phi_{c,c_1}(t_j-t_k,z-y)}{(t_j-t_k)^{1-\gamma/2}}.
\end{equation*}
Plugging the above control and \eqref{CTR_D_H21_HN_INNO_P} (or \eqref{CTR_D_H21_HN} under \A{I${}_G$})
into  \eqref{INT_DELTA_2_H} we derive:
\begin{equation*}
|\Delta H_{2}^{h,\varepsilon}(t_k,t_j,z,y)|\le \frac{\Delta_{\varepsilon,\gamma,\infty} \Phi_{c,c_1}(t_j-t_k,z-y)}{(t_j-t_k)^{1-\gamma/2}},
\end{equation*}
which together with  \eqref{CTR_DELTA_H1} and the decomposition \eqref{DECOUP_DELTA_H} completes the proof.
\end{proof} 

From Lemmas \ref{LEMME_COMP_MC} and \ref{LEMMA_DIFF_KER_MC} the proof of Theorem \ref{MTHM_M} is achieved, under \A{I${}_G$}, following the steps of Lemmas \ref{ONE_STEP_C} and \ref{LEMME_DIFF_ITE}, using the H\"older inequalities for the differences of the drift terms  for $q\in (d,+\infty)$. 

The point is that we want to justify the following inequality under \A{I${}_{P,M}$} and $q=+\infty$:
\begin{eqnarray}
\label{CTR_SER_CDM} 
|(\tilde p^h \otimes_h H^{h,(r)}-\tilde p_{\varepsilon}^h \otimes_h H_{\varepsilon}^{h,(r)})(t_i,t_j,x,y)| \\ 
\le 
(r+1)
\Delta_{\varepsilon,\gamma,\infty} \frac{\{(1\vee T^{(1-\gamma)/2})c_1\}^{r+1}\left[ \Gamma (\frac{\gamma }{2})\right] ^{r}}{\Gamma
(1+r\frac{\gamma }{2})} \frac{c^{d}}{(t_j-t_i)^{d/2}}Q_{M-(d+5+\gamma)}\left(\frac{y-x}{(t_j-t_i)^{1/2}/c}\right) (t_j-t_i)^{\frac{r \gamma}{2}}.\nonumber
\end{eqnarray}
The only delicate point, w.r.t. the analysis performed for diffusions, consists in controlling the convolutions of the densities with polynomial decay. To this end, we can adapt a technique used by Kolokoltsov \cite{kolo:00} to investigate convolutions of ``stable like" densities.
Set $m:=M-(d+5+\gamma) $ and denote for all $0\le i<j\le N,\ x\in \R^d $ by $q_m(t_j-t_i,x):=\frac{c^{d}}{(t_j-t_i)^{d/2}}Q_{M-(d+5+\gamma)}\left(\frac{x}{(t_j-t_i)^{1/2}/c}\right) $ the density with polynomial decay appearing in Lemmas \ref{LEMME_COMP_MC} and \ref{LEMMA_DIFF_KER_MC}.
Let us consider for fixed $i<k<j, (x,y)\in (\R^d)^2 $ the convolution:
\begin{equation}
\label{CONV_POL}
I_{t_k}^1(t_i,t_j,x,y):=\int_{\R^d} dz q_m(t_k-t_i,z-x)q_m(t_j-t_k,y-z).
\end{equation}
\begin{trivlist}
\item[-] If $|x-y|\le c(t_j-t_i)^{1/2} $ (diagonal regime for the parabolic scaling), it is easily seen that one of the two densities in the integral \eqref{CONV_POL} is homogeneous to $q_m(t_j-t_i,y-x) $. Namely, if $(t_k-t_i)\ge (t_j-t_i)/2 $, $q_m(t_k-t_i,z-x)\le \frac{c^{d}c_m}{(t_k-t_i)^{d/2}}\le \frac{2^{d/2}c^dc_m}{(t_j-t_i)^{d/2}} \le \tilde c q_m(t_j-t_i,y-x) $. Thus,
$$I_{t_k}^1(t_i,t_j,x,y)\le \tilde c q_m(t_j-t_i,y-x)\int_{\R^d}dz q_m(t_j-t_k,y-z)=\tilde c q_m(t_j-t_i,y-x).$$
If $(t_k-t_i)< (t_j-t_i)/2 $, the same operation can be performed taking $q_m(t_j-t_k,y-z) $ out of the integral, observing again that in that case $q_m(t_j-t_k,y-z)\le \tilde c q_m(t_j-t_i,y-x) $.
\item[-] If $|x-y|> c(t_j-t_i)^{1/2} $ (off-diagonal regime),  we introduce $A_1:=\{z\in \R^d :|x-z|\ge \frac12 |x-y|\} $, $A_2:=\{z\in \R^d:|z-y|\ge \frac12 |x-y|\} $. Every $z\in \R^d$ belongs at least to one of the $\{A_i\}_{i\in \{1,2\}}$. Let us assume w.l.o.g. that $z\in A_2 $. Then $|z-y|\ge \frac{c}2(t_j-t_i)^{1/2}\ge  \frac{c}2(t_j-t_k)^{1/2}$ so that the density $q_m(t_j-t_k,y-z) $ is itself in the off-diagonal regime. Write:
\begin{eqnarray*}
\int_{A_2}dz q_m(t_k-t_i,z-x)q_m(t_j-t_k,y-z)
\le \int_{A_2}dz q_m(t_k-t_i,z-x) \frac{c_m(t_k-t_i)^{(m-d)/2}}{|z-y|^{m}}\\
\le \frac{c_m2^m(t_j-t_i)^{(m-d)/2}}{|x-y|^m}\int_{A_2}dzq_m(t_k-t_i,z-x)
\le\bar c q_m(t_j-t_i,y-x),
\end{eqnarray*}
recalling that, under \A{I${}_{P,M}$}, $m>d$ for the last but one inequality.
The same operation could be performed on $A_1$.

We have thus established that, there exist $\bar c>1$ s.t. for all $0\le i<k<j,(x,y)\in (\R^{d})^2 $ :
\begin{equation*}
I_{t_k}^1(t_i,t_j,x,y)\le \bar c q_m(t_j-t_i,y-x).
\end{equation*}
From the controls of Lemma \ref{LEMMA_DIFF_KER_MC} and following the strategy of Lemma \ref{LEMME_DIFF_ITE}, we will be led to consider convolutions of the previous type involving $\Gamma $ functions. The above strategy thus yields \eqref{CTR_SER_CDM} by induction.
\end{trivlist}

\section*{Acknowledgements}
The authors would like to thank the two anonymous referees for their careful reading and suggestions that truly improved the previous version of this work. 
\bibliographystyle{alpha}
\bibliography{bibli}

\end{document}